\documentclass[11pt,reqno]{amsart}
\usepackage{amsmath}
\usepackage{amssymb}
\usepackage{verbatim}
\usepackage{graphicx}
\usepackage[font=footnotesize]{caption}
\usepackage{subcaption}
\usepackage{epstopdf}
\usepackage{multirow}
\usepackage{url}
\usepackage{makecell}
\usepackage{tabularx}
\usepackage{longtable}
\usepackage{graphicx}
\usepackage{tikz}
\usepackage{pgfplots}
\usepgfplotslibrary{polar}
\usepackage{xcolor}
\usepackage{pgfplots}
\usepackage{pgfplotstable}
\allowdisplaybreaks

\usepackage{extarrows}
\usepackage{bm}
\usepackage[toc]{appendix}
\usepackage[capitalise]{cleveref}
\usepackage{color}
\allowdisplaybreaks

\usepackage[top=1.5in,bottom=1.5in,left=1in,right=1in]{geometry}

\newtheorem{lemma}{Lemma}
\newtheorem{proposition}[lemma]{Proposition}

\newtheorem{theorem}[lemma]{Theorem}
\newtheorem{remark}[lemma]{Remark}

\newtheorem{example}[lemma]{Example}

\numberwithin{equation}{section}
\numberwithin{lemma}{section}

\newcommand{\N}{\mathbb{N}}    
\newcommand{\NN}{\mathbb{N}_0} 
\newcommand{\R}{\mathbb{R}}    

\newcommand{\bo}{\mathcal{O}}

\newcommand{\ka}{\textsf{k}}
\newcommand{\ba}{\textsf{b}}

\newcommand{\xa}{\textsf{x}}
\newcommand{\ya}{\textsf{y}}

\newcommand{\p}{\partial}

\newcommand{\be}{ \begin{equation} }
	\newcommand{\ee}{ \end{equation} }

\newcommand{\ind}{\Lambda}

\newcommand{\bu}{{\bm u}}
\newcommand{\bH}{{\bm H}}
\newcommand{\nab}{\nabla}
\newcommand{\bbR}{\mathbb{R}}

\newcommand{\bpsi}{{\bm \psi}}

\begin{document}
	
	\title[]{A High-Order, Pressure-Robust, and Decoupled Finite Difference Method for the Stokes Problem}
	
\author{Qiwei Feng}
\address{Department of Mathematics, University of Pittsburgh, Pittsburgh, PA 15260 USA.}
\email{qif21@pitt.edu, qfeng@ualberta.ca}

\author{Bin Han}
\address{Department of Mathematical and Statistical Sciences, University of Alberta, Edmonton, Alberta, Canada T6G 2G1.}
\email{bhan@ualberta.ca}

\author{Michael Neilan}
\address{Department of Mathematics, University of Pittsburgh, Pittsburgh, PA 15260 USA.}
\email{neilan@pitt.edu}

\thanks{Research supported in part by  the Mathematics Research Center, Department of Mathematics, University of Pittsburgh, Pittsburgh, PA, USA (Qiwei Feng),  Natural Sciences and Engineering Research Council (NSERC) of Canada under grants RGPIN-2024-04991 (Bin Han), and  the National Science Foundation grant DMS-2309425 (Michael Neilan).
}

	\makeatletter \@addtoreset{equation}{section} \makeatother

\begin{abstract}
In this paper, we consider the Stokes problem with Dirichlet boundary 
conditions and the constant kinematic viscosity $\nu$ in an axis-aligned domain $\Omega$.
We decouple
the velocity $\bu$ and pressure $p$ by deriving a novel 
biharmonic equation in $\Omega$ and  third-order boundary conditions on $\partial\Omega$.
In contrast to the fourth-order streamfunction approach, our
formulation does not require  $\Omega$ to be simply connected.
For smooth velocity fields $\bu$ in two dimensions, we explicitly construct a finite difference method (FDM) with sixth-order consistency to approximate $\bu$  at all relevant grid points: interior points, boundary side points, and boundary corner points.
The resulting scheme yields two linear systems 
$A_1u^{(1)}_h=b_1$ and $A_2u^{(2)}_h=b_2$, where $A_1,A_2$ are constant matrices, and $b_1,b_2$ are independent of the pressure $p$ and the kinematic viscosity $\nu$.  Thus, the proposed method is pressure- and viscosity-robust.
To accommodate  velocity fields with less regularity, we modify the FDM by removing singular terms in the right-hand side vectors.
Once the discrete velocity
 is computed, we apply a sixth-order finite difference operator to approximate the pressure gradient locally, without solving any additional linear systems.
In our numerical experiments, we test both smooth and non-smooth solutions
$(\bu,p)$ in a square domain, a triply connected domain, and an $L$-shaped domain in two dimensions. The results confirm sixth-order convergence of the velocity and pressure gradient in the
$\ell_\infty$-norm for smooth solutions. For non-smooth velocity fields, our method achieves the expected lower-order convergence. Moreover, the observed velocity error $\|\bu_h-\bu\|_{\infty}$
is independent of the pressure $p$  and viscosity $\nu$.

\end{abstract}	


\maketitle

\pagenumbering{arabic}

		\section{Introduction}

The Stokes and Navier-Stokes equations have numerous applications across science and engineering, particularly in fluid dynamics, meteorology, oceanography, and biomedical engineering. Precisely, they are essential for modeling airflow over aircraft, predicting weather patterns, analyzing ocean currents, and simulating blood flow. A widely utilized finite difference method (FDM) for solving these equations is the Marker-and-Cell (MAC) scheme using staggered grids, which was originally introduced in 1965 \cite{Harlow1965}. A comprehensive review of the MAC scheme can be found in \cite{McKee2008}, while an overview of finite element methods (FEMs) for the Stokes and Navier-Stokes equations is provided in \cite{JLMNR2017}. As the  proposed method in this paper is based on FDMs, we restrict our (partial) literature review to FDMs, focused on MAC-based schemes.

Throughout this paper, unless explicitly stated otherwise, the term Stokes problem/equation refers to the incompressible steady-state scaled Stokes problem/equation.
For the Stokes problem with homogeneous or non-homogeneous Dirichlet boundary conditions on a rectangular domain,
\cite{Chen2022} proposed the second-order hyper-reduced MAC scheme for the Stokes equation with the smallest viscosity  $\nu=10^{-3}$ in its examples;
\cite{Han1998} provided the
first-order error
estimate of the new mixed FEM and the MAC method; \cite{Kanschat2008,Minev2008} applied modified discontinuous Galerkin methods to derive MAC schemes; \cite{Li2015}
proved the second-order convergence rate of the MAC scheme  with non-uniform meshes and $\nu=1$; \cite{Lin2016} used a staggered finite volume element method to derive the MAC scheme with $\nu=1$ and then proved the  second-order convergence rate  over non-uniform rectangular meshes; \cite{Liu2019} described
a new  FDM with  the order of $3/2$ and $\nu=1$;  \cite{Rui2017} derived the second-order MAC scheme  using non-uniform grids  and chose $\nu=1$ in all numerical examples;
\cite{Strikwerda1984}  constructed
an iterative method of the second-order FDM  with $\nu=1$.
For the	 2D Stokes problem with general boundary conditions on a rectangle, 	 \cite{Ito2008} proposed
a fourth-order compact MAC  scheme on the staggered grid with $\nu=1$;
 and
\cite{Wang2022} deduced a
second-order augmented approach algorithm  with the smallest $\nu=10^{-6}$.
For	 2D and 3D Stokes problems with homogeneous or non-homogeneous  Dirichlet  boundary conditions on  simply connected irregular domains,  	
\cite{Chen2015} proved the  convergence rate of order $1.5$ using the triangular MAC scheme  with $\nu=1$ in all its examples;
\cite{Song2020}  (2D and 3D) presented
second-order generalized FDMs by adding an extra boundary condition from the Stokes equation with $\nu=1$; and
\cite{Strikwerda19842} (2D) developed a
second-order FDM  with $\nu=1$.

%
%
%
%
%
%
%
%
\begin{figure}[htbp]
	\centering
	\hspace{1.1cm}
	 \begin{subfigure}[b]{0.3\textwidth}
		\begin{tikzpicture}[scale = 2]
			\draw	(-1, -1) -- (-1, 1) -- (1, 1) -- (1, -1) --(-1,-1);	 
			\node (A) [scale=1] at (0,0) {$\Omega$};
		\end{tikzpicture}
	\end{subfigure}		
	 \begin{subfigure}[b]{0.3\textwidth}
		\begin{tikzpicture}[scale = 2]
			\draw	(-1, -1) -- (-1, 1) -- (1, 1) -- (1, -1) --(-1,-1);	 
			
			\draw	(-1/2, -1/2) -- (-1/2, -1/4) -- (1/4, -1/4) -- (1/4, -1/2) -- (-1/2,-1/2);	

            \draw	(1/2, -3/4) -- (1/2, 1/2) -- (3/4, 1/2) -- (3/4, -3/4) -- (1/2,-3/4);		

            \draw	(-3/4, 0) -- (-3/4, 3/4) -- (0, 3/4) -- (0, 0) -- (-3/4,0);
            	
			\node (A) [scale=1] at (0.25,0) {$\Omega$};
		\end{tikzpicture}
	\end{subfigure}
	 \begin{subfigure}[b]{0.3\textwidth}
	\begin{tikzpicture}[scale = 2]
		\draw	(-1, -1) -- (-1, 1) -- (1, 1) -- (1, 0) -- (0, 0)-- (0, -1) --(-1,-1);	
		\node (A) [scale=1] at (-0.2,0) {$\Omega$};
	\end{tikzpicture}
\end{subfigure}	
	\caption{Three examples for the domain $\Omega$ of the Stokes problem \eqref{Model:Original} in 2D.  }
	\label{fig:domain}
\end{figure}
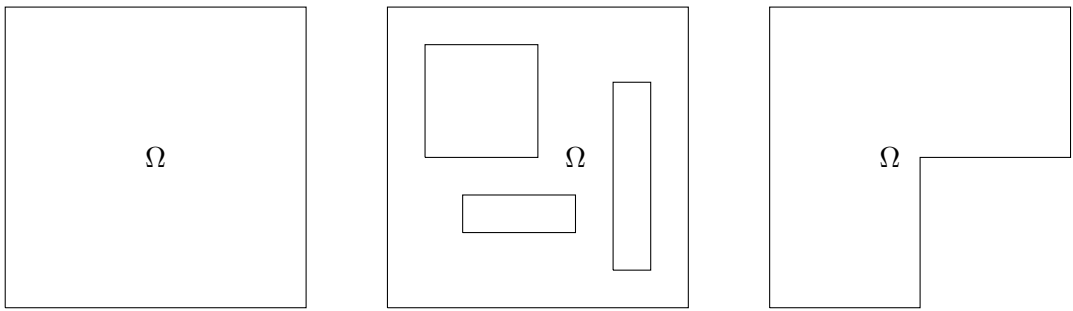

In this paper,  we consider the Stokes problem in an axis-aligned domain $\Omega \subset \bbR^d$ 
whose sides are parallel to the coordinate axes.
The domain $\Omega$ is not necessarily simply connected; see 
 Figure~\ref{fig:domain} for an illustration in two dimensions.
 The goal is to find $(\bu,p)\in  \bH^1(\Omega)\times L^2_0(\Omega)$ such that
\begin{align}\label{Model:Original}
 \hspace{3cm}	-\nu \Delta \bu + \nab p & = {\bm f} &&  \text{in }\Omega,\notag\\
 \hspace{3cm}	-\nab\cdot \bu & = \phi &&  \text{in }\Omega,\\
 \hspace{3cm}	 \bu & = {\bm g} &&   \text{on }\partial \Omega,\notag
\end{align}
where ${\bm f}:\Omega\to \bbR^d$ denotes a given external force (e.g.,
centrifugal force, buoyancy, gravity),
${\bm g}:\partial \Omega\to \bbR^d$ is the boundary velocity,
and $\phi:\Omega\to \bbR$ is a given source term.
The space $L_0^2(\Omega)$ consists 
of square-integrable functions with vanishing mean,
and $\nu>0$ is the kinematic viscosity (assumed to be constant for simplicity). The fluid velocity is denoted by $\bu$, and the pressure by $p$. When $\phi = 0$, the divergence constraint $\nabla \cdot \bu = 0$ enforces the incompressibility (i.e., mass conservation) of the fluid.

We propose a reformulation of the Stokes problem \eqref{Model:Original} that decouples the pressure from the velocity field in $\Omega \subset \mathbb{R}^d$, and further decouples these quantities on the boundary $\partial \Omega$ when $d=2$. The resulting equation for the velocity takes the form of a fourth-order PDE of biharmonic type, and we prove that this reformulation is independent of both the pressure field $p$ and the viscosity $\nu$. Using the Dirichlet boundary conditions on $\partial\Omega$ (with $\Omega\subset \R^2$) along with the momentum and continuity equations,
we decouple $\bu$ and $p$ on the boundary  by deriving third-order boundary conditions on $\bu$. Unlike the classical streamfunction formulation (see, e.g., \cite{Glowinski1979}), our approach remains valid even for non-simply connected domains.

 Based on  the two-dimensional reformulation, we construct a sixth-order FDM to compute the velocity components 
 $u_h^{(1)}$ and $u_h^{(2)}$ separately, under the assumption that $\bu$  and $p$ are smooth. 
 Specifically, we assume  $\bu$ has uniformly continuous partial derivatives of total order nine  in $\Omega$, and 
 that $p$ has uniformly continuous partial derivatives of total order seven. Since the reformulated PDEs do not involve $p$ and $\nu$, the proposed sixth-order FDM is pressure- and viscosity-robust, that is, the error between $\bu_h$  and $\bu$ is independent of the pressure $p$ and viscosity $\nu$. For $\bu$ with less regularity, the proposed method still appears to 
 converge, albeit with reduced rates as expected. 
 
In contrast to the aforementioned FDMs, our proposed scheme
achieves high-order accuracy and pressure-robustness.
To the best of our knowledge,  this is the first FDM with both of these properties.
This is made possible not by directly
discretizing the Stokes problem, but rather by developing a compact finite 
 difference stencil for the reformulated biharmonic equation,
along with discretizations of non-standard third-order boundary conditions. 
A potential drawback of the approach is that, 
although it avoids solving a saddle-point system, the resulting 
discretization of a fourth-order PDE
leads to an algebraic system
with a condition number that scales like  $\bo(h^{-4})$.


The remainder of the paper is organized as follows. 
In \cref{Reformulation}, we reformulate the Stokes problem  \eqref{Model:Original} that decouples
the velocity $\bu$ and pressure $p$ in \cref{biharmonic:propEdit}. 
In two dimensions, we decouple the velocity and pressure on the boundary 
 in \cref{three:order:pde:propEdit} and \cref{thm:all:pde:2D}. 
In \cref{FDMs:all:points}, in a 2D domain whose boundary is aligned with the coordinate axes, we  
propose a FDM to approximate the velocity $\bu=(u^{(1)},u^{(2)})$ and  gradient $(p_x,p_y)$ of the pressure $p$. Precisely, for  smooth $\bu$, we derive a sixth-order FD operator 
at interior grid points in \cref{FDMs:interior} (\cref{theorem:interior}), at boundary side points in \cref{FDMs:sides} (\cref{theorem:side:1:generalized}), at boundary corner points  in \cref{FDMs:corners} (\cref{theorem:corner:1}), and at special grid points in \cref{FDMs:special:point}. In \cref{FDMs:singular}, we modify the right-hand side of our proposed sixth-order FDM to approximate   $\bu$ with less smoothness.
In \cref{Matrix:FDMs}, we summarize the proposed FDM as a linear system,
where the constant stiffness matrix is of block diagonal and the right-hand side vector is independent
of $p$ and $\nu$. In \cref{sec:pxpy}, we use a sixth-order FD operator to approximate $(p_x,p_y)$ by the computed numerical velocity $\bu_h$.

In \cref{numerical:test}, we test examples for the model problem \eqref{Model:Original} with any $\nu>0$ in the  $l_\infty$-norm in the following three cases: (i) smooth $\bu$ and $p$ in $(-1,1)^2$ in \cref{example1};
(ii) smooth $\bu$ and $p$ with large amplitude, and singular $p$ in a triply connected domain in \cref{example2};
(iii) singular $\bu$ and smooth $p$ in an $L$-shaped domain in \cref{example3}.
The numerical experiments indicate that the FDM is stable
and converges with the expected $6$th-order accuracy in the case of smooth solutions.
In the last test, we also observe convergence with reduced rates. The numerical experiments
also illustrate the pressure-robustness of the method.
Finally, in \cref{sec:Conclu}, we summarize the main contributions of this paper
and discuss future directions.

\section{Reformulation of the Stokes Problem}\label{Reformulation}

	\subsection{The $d$-dimensional case}\label{d:dimensional}
In this section, we formally reformulate problem \eqref{Model:Original}
as a decoupled set of the fourth-order problem.
This is summarized in the following proposition.

\begin{proposition}\label{biharmonic:propEdit}
Let  $\Omega\in \R^d$ be a bounded domain ($\Omega$ is not necessarily a polytope). Suppose that 
$\bu$, ${\bm f}$, $p$,  $\phi$ satisfy \eqref{Model:Original} and are smooth. Then there holds
\begin{align}\label{eqn:BiharmonicEdit}
-\Delta^2 \bu = {\bm \psi}
\quad \mbox{with}\quad
{\bm \psi}:=	\Delta (\nab \phi) +\nu^{-1}[\Delta {\bm f}-\nab (\nab\cdot {\bm f})].
\end{align}
Furthermore, the quantities ${\bm \psi}$ and $\nu^{-1} \Delta {\bm f}-\nu^{-1}\nab (\nab \cdot {\bm f})$ are independent of both the pressure $p$ and the constant kinematic viscosity $\nu$ due to
\begin{equation}\label{f:nop}
\nu^{-1}[\Delta {\bm f}-\nab (\nab \cdot {\bm f})]=-\Delta^2 \bu+\nab (\nab \cdot \Delta \bu).
\end{equation}
\end{proposition}

\begin{proof}
From $-\nu \Delta \bu + \nab p = {\bm f}$ in \eqref{Model:Original},
we have $\nab p={\bm f}+\nu \Delta \bu $ and hence
\begin{align*}
\nab \Delta p =\nab ( \nab \cdot \nab p) = \nab \big(\nab\cdot {\bm f} +\nu \nab \cdot \Delta \bu\big) = \nab (\nab \cdot {\bm f}) -\nu \Delta \nab \phi,
\end{align*}
where we used $-\nab\cdot \bu = \phi$.
Using ${\bm f}=-\nu \Delta \bu + \nab p$ and the above identity, we deduce
\begin{align*}
\nu^{-1} \Delta {\bm f} = \nu^{-1} \Delta (-\nu \Delta \bu+\nab p) = -\Delta^2 \bu+\nu^{-1} \nab (\nab \cdot {\bm f}) -\Delta \nab \phi,
\end{align*}
where we used the simple fact $\Delta (\nab p)=\nab \Delta p$.
Rearranging the terms in this identity yields \eqref{eqn:BiharmonicEdit}.

On the other hand, we deduce from ${\bm f}=-\nu \Delta \bu + \nab p$ that
\[
\nu^{-1}\Delta {\bm f}=-\Delta^2 \bu+\nu^{-1} \Delta \nab p
\quad \mbox{and}\quad
\nu^{-1} \nab (\nab \cdot {\bm f})=-\nab (\nab \cdot \Delta \bu)+\nu^{-1} \nab (\nab\cdot \nab p).
\]
Because
$\Delta {\nab p}-\nab (\nab\cdot {\nab p})=\Delta {\nab p}- \nab \Delta  p=0$,
we conclude that \eqref{f:nop} holds.
\end{proof}

\subsection{The 2D case}\label{2D:case} In this section, we restrict our problem  \eqref{Model:Original}
to a polygonal domain $\Omega\subset \bbR^2$ ($\Omega$ is not necessarily an axis-aligned domain).
We formally reformulate problem \eqref{Model:Original} on $\partial \Omega$
 as third-order PDEs in the following  proposition.
\begin{proposition}\label{three:order:pde:propEdit}
Suppose that $\bu=(u^{(1)},u^{(2)})$, ${\bm f}=(f^{(1)},f^{(2)})$, ${\bm g}=(g^{(1)},g^{(2)})$,  $p$, $\phi$ satisfy \eqref{Model:Original}
and	are smooth in the closure of a polygonal domain $\Omega\subset  \R^2$ ($\Omega$ is not necessarily an axis-aligned domain). Let $\Gamma$  be a straight-line segment on the boundary $\partial\Omega$ which can be parametrized as
\be \label{parametrized}
x(t)={\textup\ka}_1t+{\textup\ba}_1, \quad y(t)={\textup\ka}_2t+{\textup\ba}_2, \qquad {\textup\ka}_1, {\textup\ka}_2, {\textup\ba}_1, {\textup\ba}_2\in \R, \qquad (x(t),y(t)) \in \Gamma \subset \Omega,
\ee
where $|{\textup\ka}_1|+|{\textup\ka}_2|\ne 0$.
Then there holds
\be
	 \begin{split}\label{bd:eqn:1:three:order:Edit} &{\textup\ka}_1^2({\textup\ka}_1-3{\textup\ka}_2)u^{(1)}_{xxx}+{\textup\ka}_1(2{\textup\ka}_1^2+3{\textup\ka}_1{\textup\ka}_2-3{\textup\ka}_2^2)u^{(1)}_{xxy}+	 {\textup\ka}_2^2(3{\textup\ka}_1-{\textup\ka}_2)u^{(1)}_{xyy}+({\textup\ka}_1^3+{\textup\ka}_2^3 )u^{(1)}_{yyy}\\
	&\quad  =3{\textup\ka}_1^2{\textup\ka}_2\phi_{xx}+{\textup\ka}_1(3{\textup\ka}_2^2-{\textup\ka}_1^2)\phi_{xy} +{\textup\ka}_2^3\phi_{yy} +\nu^{-1}{\textup\ka}_1^3( f^{(2)}_{x} - f^{(1)}_{y})+\rho_1+\rho_2 \quad \text{on} \quad \Gamma,
	\end{split}
	\ee
	and
	\be
	 \begin{split}\label{bd:eqn:2:three:order:Edit}
		 &({\textup\ka}_1^3+{\textup\ka}_2^3)u^{(2)}_{xxx}+{\textup\ka}_1^2(3{\textup\ka}_2-{\textup\ka}_1)u^{(2)}_{xxy}+	 {\textup\ka}_2(2{\textup\ka}_2^2+3{\textup\ka}_1{\textup\ka}_2-3{\textup\ka}_1^2)u^{(2)}_{xyy}+{\textup\ka}_2^2({\textup\ka}_2-3{\textup\ka}_1 )u^{(2)}_{yyy}\\
		&\quad  ={\textup\ka}_1^3\phi_{xx}+{\textup\ka}_2(3{\textup\ka}_1^2-{\textup\ka}_2^2)\phi_{xy} +3{\textup\ka}_1{\textup\ka}_2^2 \phi_{yy} +\nu^{-1}{\textup\ka}_2^3( f^{(1)}_{y} -f^{(2)}_{x})  +\rho_1+\rho_2 \quad \text{on} \quad \Gamma,
	\end{split}
	\ee
where
\[
\rho_r:=\sum_{i=0}^3{3 \choose i}  {\textup\ka}_1^{3-i}{\textup\ka}_2^{i}\frac{\p^3 g^{(r)}}{\p x^{3-i} \p y^{i}}, \qquad r=1,2.
\]
Furthermore, the right-hand sides of \eqref{bd:eqn:1:three:order:Edit} and \eqref{bd:eqn:2:three:order:Edit} are completely independent of both the pressure $p$ and the constant kinematic viscosity $\nu$ in \eqref{Model:Original} due to
\begin{equation}\label{indepen:p:v}
 \nu^{-1} [f^{(1)}_y- f^{(2)}_x]=\Delta [u^{(2)}_{x}-u^{(1)}_{y}].
\end{equation}
\end{proposition}
\begin{proof}
We only prove  \eqref{bd:eqn:1:three:order:Edit} and \eqref{indepen:p:v}, as \eqref{bd:eqn:2:three:order:Edit} can be proved similarly.
It follows directly from \eqref{Model:Original} and \eqref{parametrized} that on $\Gamma$, we have
\begin{align}
	-u^{(1)}_{xxx}-	u^{(2)}_{xxy} &= \phi_{xx},\label{ident:1}\\
	-u^{(1)}_{xxy}-	u^{(2)}_{xyy} &= \phi_{xy},\label{ident:2}\\
	-u^{(1)}_{xyy}-	u^{(2)}_{yyy} &= \phi_{yy},\label{ident:3}\\
	-\nu u^{(2)}_{xxx}	-\nu u^{(2)}_{xyy} + p_{xy} & = f^{(2)}_x,\label{ident:4}\\
	-\nu u^{(1)}_{xxy}	-\nu u^{(1)}_{yyy} + p_{xy} & = f^{(1)}_y,\label{ident:5}\\
u_{ttt}^{(1)}= g_{ttt}^{(1)}\Rightarrow \ka_1^3u^{(1)}_{xxx}+3\ka_1^2\ka_2u^{(1)}_{xxy}+3\ka_1\ka_2^2u^{(1)}_{xyy}+\ka_2^3u^{(1)}_{yyy} & =\rho_1, \label{ident:6}\\
u_{ttt}^{(2)} = g_{ttt}^{(2)} \Rightarrow \ka_1^3u^{(2)}_{xxx}+3\ka_1^2\ka_2u^{(2)}_{xxy}+3\ka_1\ka_2^2u^{(2)}_{xyy}+\ka_2^3u^{(2)}_{yyy} &=\rho_2. \label{ident:7}
\end{align}
Then
$ 3\ka_1^2\ka_2 \times \eqref{ident:1}+\ka_1(3\ka_2^2-\ka_1^2)  \times \eqref{ident:2} +\ka_2^3 \times \eqref{ident:3} +\nu^{-1}\ka_1^3 \times (  \eqref{ident:4} - \eqref{ident:5}) +\eqref{ident:6}   + \eqref{ident:7}$ yields
\begin{align*}
&3{\textup\ka}_1^2{\textup\ka}_2\phi_{xx}+{\textup\ka}_1(3{\textup\ka}_2^2-{\textup\ka}_1^2)\phi_{xy} +{\textup\ka}_2^3\phi_{yy} +\nu^{-1}{\textup\ka}_1^3( f^{(2)}_{x} - f^{(1)}_{y}) + \rho_1+\rho_2\\	
& =3\ka_1^2\ka_2 (	-u^{(1)}_{xxx}-	 u^{(2)}_{xxy})+\ka_1(3\ka_2^2-\ka_1^2)  (-u^{(1)}_{xxy}-	u^{(2)}_{xyy}) +\ka_2^3 (	-u^{(1)}_{xyy}-	 u^{(2)}_{yyy}) \\
&\quad +\nu^{-1}\ka_1^3(  	-\nu u^{(2)}_{xxx}	-\nu u^{(2)}_{xyy}  +\nu u^{(1)}_{xxy}	+\nu u^{(1)}_{yyy} )+\ka_1^3u^{(1)}_{xxx}+3\ka_1^2\ka_2u^{(1)}_{xxy} \\
&\quad+3\ka_1\ka_2^2u^{(1)}_{xyy} +\ka_2^3u^{(1)}_{yyy}+\ka_1^3u^{(2)}_{xxx}+3\ka_1^2\ka_2u^{(2)}_{xxy}+3\ka_1\ka_2^2u^{(2)}_{xyy}+\ka_2^3u^{(2)}_{yyy}\\
& =(\ka_1^3-3\ka_1^2\ka_2)u^{(1)}_{xxx}+(-3\ka_1\ka_2^2+\ka_1^3+\ka_1^3+3\ka_1^2\ka_2)  u^{(1)}_{xxy}+(3\ka_1\ka_2^2-\ka_2^3) u^{(1)}_{xyy}+	(\ka_1^3+\ka_2^3 ) u^{(1)}_{yyy}\\
& \quad + (\ka_1^3-\ka_1^3)u^{(2)}_{xxx}+(3\ka_1^2\ka_2 -3\ka_1^2\ka_2)u^{(2)}_{xxy}+ (-3\ka_1\ka_2^2+\ka_1^3-\ka_1^3+3\ka_1\ka_2^2)  	 u^{(2)}_{xyy} +(\ka_2^3-\ka_2^3) 	 u^{(2)}_{yyy} \\
& =(\ka_1^3-3\ka_1^2\ka_2)u^{(1)}_{xxx}+(2\ka_1^3-3\ka_1\ka_2^2+3\ka_1^2\ka_2)  u^{(1)}_{xxy}+(3\ka_1\ka_2^2-\ka_2^3) u^{(1)}_{xyy}+	(\ka_1^3+\ka_2^3 ) u^{(1)}_{yyy}, \quad \text{on} \quad \Gamma.
\end{align*}
So \eqref{bd:eqn:1:three:order:Edit} is proved.

Next, using  $-\nu \Delta \bu+\nab p={\bm f}$, we have
\[
\nu^{-1}f^{(1)}_y-\nu^{-1} f^{(2)}_x=
[- \Delta u^{(1)}_{y}+\nu^{-1} p_{xy}]-
[- \Delta u^{(2)}_{x}+\nu^{-1} p_{yx}]
=\Delta [u^{(2)}_{x}-u^{(1)}_{y}],
\]
which only depends on $\bu$. This proves \eqref{indepen:p:v}.
\end{proof}

If $\p\Omega$ only contains edges parallel to the coordinate axis, i.e.
if $(\ka_1,\ka_2)=(0,1)$ in \eqref{parametrized} for  $x=\text{const}$ and  $(\ka_1,\ka_2)=(1,0)$ in \eqref{parametrized} for  $y=\text{const}$, then \cref{biharmonic:propEdit,three:order:pde:propEdit}
reduce to the following reformulation.

\begin{theorem}\label{thm:all:pde:2D}
Assume that the boundary of a domain $\Omega$ has a decomposition
$\p\Omega = \big(\cup_{m=1}^{N_x} \overline{\Gamma_m^x}\big)\cup \big(\cup_{n=1}^{N_y} \overline{\Gamma_n^y}\big)$,
where $\Gamma_m^x = \{(\alpha_m,y):\ a^x_m< y<b^x_m\}$ and $\Gamma_n^y = \{(x,\beta_n):\ a^y_n< x< b_n^y\}$
with $\alpha_m,a^x_m, b^x_m, \beta_n, a^y_n, b^y_n \in \bbR$ and $N_x,N_y\in \N$. That is, the variable $x$ is constant on $\Gamma_m^x$, and the variable
$y$ is constant on $\Gamma_n^y$.
Suppose that  $\bu=(u^{(1)},u^{(2)})$,  ${\bm f}=(f^{(1)},f^{(2)})$,  ${\bm g}=(g^{(1)},g^{(2)})$, 	$\phi$ and $p$ satisfy \eqref{Model:Original} and are smooth on $\overline{ \Omega}$.
Then the following PDEs hold
\begin{equation}\label{eqn:FourthOrderPDE}
\left\{
\begin{array}{ll}
 -\Delta^2 u^{(1)} = \psi^{(1)} \qquad &\text{in }\Omega,\\
 u^{(1)} =g^{(1)} \quad &\text{on }\p\Omega,\\
 u^{(1)}_{xyy} = \varphi^{(1)}_1    &\text{on }\Gamma_m^x,\\
u^{(1)}_{yyy}	+2u^{(1)}_{xxy}  =\varphi^{(1)}_2  &\text{on }\Gamma_n^y,
\end{array}\right.
\qquad
\left\{\begin{array}{ll}
- \Delta^2 u^{(2)} = \psi^{(2)} \qquad &\text{in }\Omega,\\
 u^{(2)} =g^{(2)} \quad &\text{on }\p\Omega,\\
 u^{(2)}_{xxx}+ 2u^{(2)}_{xyy} = \varphi^{(2)}_1   &\text{on }\Gamma_m^x,\\
u^{(2)}_{xxy}    =\varphi^{(2)}_2  &\text{on }  \Gamma_n^y,
\end{array}\right.
\end{equation}
for $m=1,\ldots,N_x$ and $n=1,\ldots,N_y$,
where $(\psi^{(1)},\psi^{(2)})={\bm \psi}=\Delta (\nab \phi) + \nu^{-1}[ \Delta {\bm f}- \nab (\nab\cdot {\bm f})]$ is defined in \eqref{eqn:BiharmonicEdit}, and
		\begin{equation}	 \label{RHS:2d}
		\begin{split}
		& \varphi^{(1)}_1:=-\phi_{yy}-g^{(2)}_{yyy},  \qquad \varphi^{(1)}_2:=
-\nu^{-1}[
f^{(1)}_{y}-f^{(2)}_{x}]- \phi_{xy}+g^{(2)}_{xxx},\qquad   \\
		&  \varphi^{(2)}_2:=-\phi_{xx}-g^{(1)}_{xxx}, \qquad \varphi^{(2)}_1:=
\nu^{-1}[f^{(1)}_{y}-  f^{(2)}_{x}]
-  \phi_{xy}+g^{(1)}_{yyy}.
		\end{split}
		\end{equation}	
	Importantly,  the right-hand side functions $\bpsi$ and $\varphi^{(r)}_j$ with $r,j\in \{1,2\}$ are all completely independent of both the pressure $p$ and the constant kinematic viscosity $\nu$ in \eqref{Model:Original} due to
\begin{equation}\label{nop:all}
\nu^{-1} [\Delta {\bm f}- \nab (\nab \cdot {\bm f})]=- \Delta^2 \bu+\nab (\nab \cdot \Delta \bu),
\qquad \nu^{-1} [f^{(1)}_y- f^{(2)}_x]=\Delta [u^{(2)}_{x}-u^{(1)}_{y}].
\end{equation}
\end{theorem}
\begin{proof}
The identities $-\Delta^2 u^{(1)}=\psi^{(1)}$, $-\Delta^2 u^{(2)}=\psi^{(2)}$, and the independence of $\bpsi$ from $p$ follow directly from
\cref{biharmonic:propEdit}.

 \eqref{bd:eqn:1:three:order:Edit} with $(\ka_1,\ka_2)=(0,1)$ and  $(\ka_1,\ka_2)=(1,0)$ results in the third-order PDEs  $u^{(1)}_{xyy} = \varphi^{(1)}_1$  and
 $u^{(1)}_{yyy}	+2u^{(1)}_{xxy}  =\varphi^{(1)}_2$, respectively.  Likewise, \eqref{bd:eqn:2:three:order:Edit} with $(\ka_1,\ka_2)=(0,1)$ and  $(\ka_1,\ka_2)=(1,0)$ results in the third-order PDEs  $ u^{(2)}_{xxx}+ 2u^{(2)}_{xyy} = \varphi^{(2)}_1$  and
 $u^{(2)}_{xxy}    =\varphi^{(2)}_2$, respectively.

The first identity of \eqref{nop:all} has already been proved in \eqref{f:nop}, and the second identity in 
\eqref{nop:all} is a restatement of
\eqref{indepen:p:v}.
\end{proof}	
\begin{remark}
	
 Note that, similar to the well-known streamfunction formulation (see, e.g., \cite{Glowinski1979})
		the PDEs \eqref{eqn:BiharmonicEdit} and \eqref{eqn:FourthOrderPDE}
		are fourth-order problems involving the biharmonic operator.
		However these formulations are clearly distinct, as \eqref{eqn:BiharmonicEdit}
		and \eqref{eqn:FourthOrderPDE} are formulated solely in terms of the velocity.
		Furthermore, unlike the streamfunction formulation, we do not require
		the domain $\Omega$ to be simply connected.
\end{remark}

\begin{remark}
After determining the unknown velocity fields $u^{(1)}$ and $u^{(2)}$, we can determine  the gradient of the pressure $p$ through $\nab p={\bm f}+\nu \Delta u$ in \eqref{Model:Original}.
\end{remark}

	\section{FDM on the Uniform Cartesian Grid for the 2D Stokes Problem}\label{FDMs:all:points}

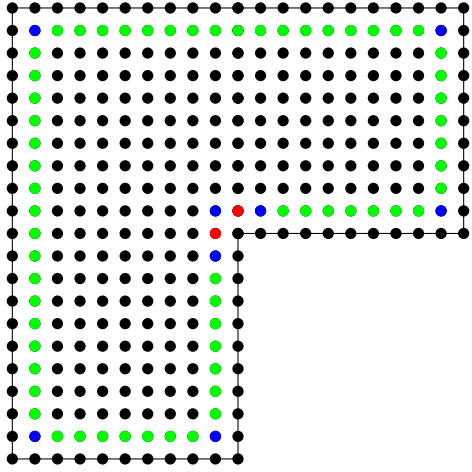
\begin{figure}[htbp]
	\centering
	
	\begin{tikzpicture}[scale = 0.6]
		\draw	(-5, -5) -- (-5, 5) -- (5, 5) -- (5, 0) -- (0, 0)-- (0, -5) --(-5,-5);	
		
		
		\foreach \y in {-10,...,10}
		\foreach \x in {-10,...,0}
		{\node at (0.5*\x,0.5*\y)[circle,fill,inner sep=1.5pt,color=black]{};}
		
		\foreach \y in {0,...,10}
		\foreach \x in {0,...,10}
		{\node at (0.5*\x,0.5*\y)[circle,fill,inner sep=1.5pt,color=black]{};}

		\foreach \y in {-8,...,8}
		{\node at (-4.5,0.5*\y)[circle,fill,inner sep=1.5pt,color=green]{};}
		
		\foreach \y in {-8,...,-2}
		{\node at (-0.5,0.5*\y)[circle,fill,inner sep=1.5pt,color=green]{};}	
		
		\foreach \y in {2,...,8}
		{\node at (4.5,0.5*\y)[circle,fill,inner sep=1.5pt,color=green]{};}

		\foreach \x in {-8,...,-2}
		{\node at (0.5*\x,-4.5)[circle,fill,inner sep=1.5pt,color=green]{};}
		
		\foreach \x in {-8,...,8}
		{\node at (0.5*\x,4.5)[circle,fill,inner sep=1.5pt,color=green]{};}	
		
		\foreach \x in {2,...,8}
		{\node at (0.5*\x,0.5)[circle,fill,inner sep=1.5pt,color=green]{};}	
		
		\node at (-0.5,0.5)[circle,fill,inner sep=1.5pt,color=blue]{};
		\node at (-0.5,-0.5)[circle,fill,inner sep=1.5pt,color=blue]{};
		\node at (0.5,0.5)[circle,fill,inner sep=1.5pt,color=blue]{};
		\node at (4.5,0.5)[circle,fill,inner sep=1.5pt,color=blue]{};
		\node at (4.5,4.5)[circle,fill,inner sep=1.5pt,color=blue]{};
		\node at (-4.5,4.5)[circle,fill,inner sep=1.5pt,color=blue]{};
		\node at (-4.5,-4.5)[circle,fill,inner sep=1.5pt,color=blue]{};
		\node at (-0.5,-4.5)[circle,fill,inner sep=1.5pt,color=blue]{};
		
		\node at (0,0.5)[circle,fill,inner sep=1.5pt,color=red]{};
		\node at (-0.5,0)[circle,fill,inner sep=1.5pt,color=red]{};
	\end{tikzpicture}
	\caption{Illustration of the decomposition of the grid points.
		Points in $\Gamma^x_{m,h}$ or $\Gamma^y_{n,h}$ are depicted
		in green, grid points $(\xa_i,\ya_j)$ satisfying $(\xa_i\pm h,\ya_j\pm h)$ or $(\xa_i\pm h,\ya_j\mp h)\in \Omega^c$
		are depicted in blue, and grid points $(\xa_i,\ya_j)$ satisfying $(\xa_i\pm h,\ya_j)\in \Omega^c$ and $(\xa_i\pm 2h,\ya_j)\in \Gamma_n^y$
		or $(\xa_i ,\ya_j\pm h)\in \Omega^c$ and $(\xa_i,\ya_j\pm 2h)\in \Gamma_m^x$ are depicted in red.
	}
	\label{FDM:corner:special}
\end{figure}	
In this section, we construct a finite difference method (FDM)
based on the fourth-order reformulation of the Stokes problem \eqref{eqn:FourthOrderPDE}.
To this end, we assume that the domain satisfies the conditions in Theorem \ref{thm:all:pde:2D}
(i.e., the boundary of $\Omega$ aligns with the coordinate axes). 	 To express the formula of our FDM concisely, we assume $\bu=0$ (i.e., $\bm g=0$) on $\partial \Omega$ in \eqref{Model:Original}
(no-slip boundary condition) in this section. However, it is straightforward to extend the proposed FDM to the non-homogeneous Dirichlet boundary condition, and in all numerical examples in this paper, we test $\bu|_{\partial \Omega} \ne \bm 0$. For simplicity in presentation,
we construct a uniform partition of the domain based on a single discretization parameter $h>0$ as follows:
\begin{equation}
\overline\Omega_h :=  \overline \Omega \cap \left(h \mathbb{Z}^2\right).
\end{equation}
Recall that $\p\Omega = \big(\cup_{m=1}^{N_x} \overline{\Gamma_m^x}\big)\cup \big(\cup_{n=1}^{N_y} \overline{\Gamma_n^y}\big)$,
where $\Gamma_m^x: = \{(\alpha_m,y):\ a^x_m< y< b^x_m\}$ and $\Gamma_n^y: = \{(x,\beta_n):\ a^y_n< x< b_n^y\}$
with $\alpha_m,a^x_m, b^x_m, \beta_n, a^y_n, b^y_n \in \bbR$ and $N_x,N_y\in \N$.
We assume, that for each $1\le m\le N_x$ and $1\le n\le N_y$, there
exists $i_m,j_n\in \mathbb{Z}$ such that $\alpha_m = i_m h$ and
$\beta_n = j_n h$, i.e., we assume the grid points intersect the boundary edges (see Figure \ref{FDM:corner:special}).
Define the set of nodes
in a neighborhood of $\Gamma_{m}^x$ and $\Gamma_n^y$ respectively as (cf.~green points in Figure \ref{FDM:corner:special})
\be\label{Gamma:bd}
\begin{split}
\Gamma_{m,h}^x &:= \{(\alpha_m\pm h,\ya_j):\ a_m^x+2h\le \ya_j\le b_m^x-2h\},\\
\Gamma_{n,h}^y &:= \{(\xa_i,\beta_n\pm h):\ a_n^y+2h\le \xa_i\le b_n^y-2h\},
\end{split}
\ee
where the signs in the above definitions are chosen such that the grid points lie in the interior of the domain.
Denote by $\sigma^x_m,\sigma^y_n\in \{-1,1\}$ such that $\alpha_m+2\sigma^x_m h,\beta_n+2\sigma^y_n h\in \overline\Omega$.
Finally, we set
\[
\Omega_h :=\{(\xa_i,\ya_j) :\   (\xa_i+kh,\ya_j+\ell h)\in \overline{\Omega} \quad \text{ for any } k,\ell=0,\pm1,\pm2 \},
\]
to be the set of interior grid points (cf.~black points in Figure \ref{FDM:corner:special}). Finally, 
we define the set of corner points
\begin{equation}\label{Omega:c}
\Omega^c := \{(\alpha_m,\beta_n):\ 1\le m\le N_x,\ 1\le n\le N_y\}\cap \overline{\Omega}.
\end{equation}

We define  $\bu_h=(u^{(1)}_h,u^{(2)}_h)$  to be the numerical  velocity  approximation using our proposed FDM with the uniform mesh size $h$.
To describe our sixth-order FDM clearly at all grid points, we derive a sixth-order 25-point FD operator at  interior grid points in \cref{FDMs:interior},  a sixth-order 15-point FD operator at  boundary side points in \cref{FDMs:sides},  a sixth-order 2-point FD operator at  boundary corner points in \cref{FDMs:corners}, and an arbitrary high-order 1-point FD operator at special grid points in \cref{FDMs:special:point}. In \cref{FDMs:singular}, we modify the right-hand side of the sixth-order FDM to approximate the velocity $\bu$ with less regularity.
These sixth-order FDMs are constructing using the ideas  in \cite{Feng2024}.
However,  to simplify the presentation, we 
simply provide the stencils of the FDMs directly and use them to prove the consistency order.
%

\subsection{A sixth-order FD operator for smooth $\bu$ at interior grid points}\label{FDMs:interior}
We present a 25-point FD operator with  the sixth-order of consistency at the interior grid point in the following lemma. 
\begin{lemma}\label{theorem:interior}
We define the following 25-point FD operator on the set of interior grid points $\Omega_h$:
\be\label{FDMs:u:interior}
\mathcal{L}_h (\bu_h)_{i,j} :=\sum_{k,\ell=-2}^2 C_{k,\ell} (\bu_h)_{i+k,j+\ell},
\ee
where $C_{0,0}:=-13$,
\be
\label{Ckl:interior}
\begin{split}
& C_{-2,-2}=C_{-2,2}=C_{2,-2}=C_{2,2}:=-\tfrac{1}{36}, \qquad C_{-2,0}=C_{2,0}=C_{0,-2}=C_{0,2}:=-\tfrac{1}{2} ,\\
& C_{-2,-1}=C_{-2,1}= C_{2,-1}=C_{2,1}=C_{-1,-2}=C_{-1,2}=C_{1,-2}=C_{1,2}:=-\tfrac{2}{9},\\
& C_{-1,-1}=C_{-1,1}=C_{1,-1}=C_{1,1}:=\tfrac{2}{9}, \qquad  C_{-1,0}=C_{1,0}=C_{0,-1}=C_{0,1}:=4.
\end{split}
\ee	
We also set
\be\label{bpsi:h}
(\bpsi_h)_{i,j} :=
\left( {\bm\psi}  + \tfrac{h^2}{6} \Delta{\bm\psi}   +h^4\Big(\tfrac1{80} \Delta^2 \bpsi + \tfrac1{90} \tfrac{\p^4 \bpsi}{\p x^2 \p y^2}
\Big)\right)\Big|_{x={\textup\xa}_i,y={\textup\ya}_j}.\\
\ee
Then there holds $\left|\left(h^{-4}\mathcal{L}_h (\bu)- \bpsi_h\right)_{i,j}\right|=\bo(h^6)$ at $({\textup\xa}_i,{\textup\ya}_j) \in \Omega_h$ provided  $\bu$ is smooth.
\end{lemma}
\begin{proof}
Using a Taylor approximation at an interior grid point $(\xa_i,\ya_j)$, we have
\be \label{u:original}
\bu(x_i+x,y_j+y)
=
\sum_{(m,n)\in \ind_{M+1}}
\frac{\partial^{m+n}\bu(\xa_i,\ya_j) }{\partial x^m \partial y^n}\frac{x^my^n}{m!n!}+\bo(h^{M+2}), \qquad x,y\in \{0,\pm h, \pm 2h\},
\ee
where
\be \label{ind}
\ind_{M+1}:=\{ (m,n)\in \NN^2 \; : \;  m+n\le M+1 \}, \qquad M+1\in \NN,\qquad \NN:=\N\cup\{0\}.
\ee
For the sake of readability, we omit the subscripts and use the following notation in this proof:
\[
\frac{\partial^{m+n}\bu }{\partial x^m \partial y^n}:=\frac{\partial^{m+n}\bu(\xa_i,\ya_j) }{\partial x^m \partial y^n}, \qquad \frac{\partial^{m+n} \bpsi}{\partial x^m \partial y^n}:=\frac{\partial^{m+n}\bpsi(\xa_i,\ya_j) }{\partial x^m \partial y^n}.
\]
%

A calculation shows, using \eqref{FDMs:u:interior}--\eqref{Ckl:interior}  and \eqref{u:original} with $M=8$,
	 \begin{equation}\label{interior:plug:taylor}
	\begin{split}
	\frac1{h^4}\mathcal{L}_h (\bu)_{i,j}
	& =-\Big( \tfrac{ \partial^{4}\bu }{\partial x^4} + 2\tfrac{\partial^{4}\bu }{\partial x^2 \partial y^2}  +\tfrac{\partial^{4}\bu }{\partial y^4} \Big)
	- \tfrac{h^2}{6}\Big(  \tfrac{\partial^{6}\bu }{\partial x^6} +3 \tfrac{ \partial^{6}\bu }{\partial x^4 \partial y^2}  + 3\tfrac{\partial^{6}\bu }{\partial x^2 \partial y^4}  +\tfrac{\partial^{6}\bu }{\partial y^6} \Big) 	 \\
	&\quad-h^4\Big(  \tfrac{1}{80}\tfrac{\partial^{8}\bu }{\partial x^8}  +  \tfrac{11}{180}\tfrac{\partial^{8}\bu }{\partial x^6 \partial y^2} + \tfrac{7}{72}\tfrac{\partial^{8}\bu }{\partial x^4 \partial y^4}+  \tfrac{11}{180}\tfrac{\partial^{8}\bu }{\partial x^2 \partial y^6}  +\tfrac{1}{80} \tfrac{\partial^{8}\bu }{\partial y^8} \Big) +\bo(h^{6})\\
%
%
& = \bpsi +\tfrac{h^2}{6} \Delta \bpsi\\
&\quad-h^4\Big(  \tfrac{1}{80}\tfrac{\partial^{8}\bu }{\partial x^8}  +  \tfrac{11}{180}\tfrac{\partial^{8}\bu }{\partial x^6 \partial y^2} + \tfrac{7}{72}\tfrac{\partial^{8}\bu }{\partial x^4 \partial y^4}+  \tfrac{11}{180}\tfrac{\partial^{8}\bu }{\partial x^2 \partial y^6}  +\tfrac{1}{80} \tfrac{\partial^{8}\bu }{\partial y^8} \Big) +\bo(h^{6}).
	\end{split}
	\end{equation}
Next, we use a simple calculation and $-\Delta^2 \bu = \bpsi$ to write
\begin{align*}
  &\tfrac{1}{80}\tfrac{\partial^{8}\bu }{\partial x^8}  +  \tfrac{11}{180}\tfrac{\partial^{8}\bu }{\partial x^6 \partial y^2} + \tfrac{7}{72}\tfrac{\partial^{8}\bu}{\partial x^4 \partial y^4}+  \tfrac{11}{180}\tfrac{\partial^{8}\bu }{\partial x^2 \partial y^6}  +\tfrac{1}{80} \tfrac{\partial^{8}\bu }{\partial y^8}\\
  &\qquad = \tfrac1{80} \Delta^4 \bu + \tfrac1{90} \tfrac{\partial^{8}\bu }{\partial x^6 \partial y^2}
  + \tfrac{1}{45} \tfrac{\partial^{8}\bu}{\partial x^4 \partial y^4}+\tfrac1{90}\tfrac{\partial^{8}\bu }{\partial x^2 \partial y^6}\\
  &\qquad = \tfrac1{80} \Delta^4 \bu+ \tfrac1{90} \tfrac{\partial^4}{\p x^2 \p y^2} \Delta^2 \bu
 = -\tfrac1{80}  \Delta^2 \bpsi - \tfrac1{90} \tfrac{\partial^4}{\p x^2 \p y^2} \bpsi.
\end{align*}
Combining this identity with \eqref{interior:plug:taylor} yields
	\begin{align*}
	\frac1{h^4}\mathcal{L}_h (\bu)_{i,j}
	 &=  \bpsi + \tfrac{h^2}{6} \Delta \bpsi + h^4 \left(\tfrac1{80}  \Delta^2 \bpsi + \tfrac1{90} \tfrac{\partial^4}{\p x^2 \p y^2} \bpsi\right)+\bo(h^6)
	  = (\bpsi_h)_{i,j}+\bo(h^6).
	\end{align*}
This proves the result.
\end{proof}
%

\subsection{A sixth-order  FD operator for smooth $\bu$  at  boundary side points}\label{FDMs:sides}
In this section, we propose a 15-point FD operator with the sixth-order of consistency at
boundary side points.
To facilitate the presentation, we define the coefficients
\begin{align}
	& C^{(1)}_{0,-2}=C^{(1)}_{0,2}:=-2,&&  C^{(1)}_{1,-2}=C^{(1)}_{1,2}:=\tfrac{17}{20}, &&  C^{(1)}_{2,-2}=C^{(1)}_{2,2}:=-\tfrac{2}{45}, \notag \\
	& C^{(1)}_{0,-1}=C^{(1)}_{0,1}:=11,   && C^{(1)}_{1,-1}=C^{(1)}_{1,1}:=-\tfrac{49}{10}, && C^{(1)}_{2,-1}=C^{(1)}_{2,1}:=\tfrac{23}{45},\label{Ckl:1:bd}\\
	& C^{(1)}_{0,0}:=-18,  && C^{(1)}_{1,0}:=\tfrac{81}{10}, && C^{(1)}_{2,0}:=-\tfrac{14}{15},\notag
\end{align}	
\begin{align}
	& C^{(2)}_{0,-2}=C^{(2)}_{0,2}:=\tfrac{107}{120}, &&C^{(2)}_{1,-2}=C^{(2)}_{1,2}:=\tfrac{7}{30},  && C^{(2)}_{2,-2}=C^{(2)}_{2,2}:=\tfrac{1}{24}, \notag  \\
	 &C^{(2)}_{0,-1}=C^{(2)}_{0,1}:=-\tfrac{67}{15}, &&  C^{(2)}_{1,-1}=C^{(2)}_{1,1}:=-\tfrac{1}{30},&&C^{(2)}_{2,-1}=C^{(2)}_{2,1}:=\tfrac{1}{5},  \label{Ckl:2:bd}\\
	 &C^{(2)}_{0,0}:=\tfrac{203}{20},  &&C^{(2)}_{1,0}:=-\tfrac{17}{5}, && C^{(2)}_{2,0}:=\tfrac{31}{60},\notag
\end{align}	
and the functions
\begin{equation}\label{rhos}
\begin{split}
\varphi^{(1)}_{1,h}
&:=	 \varphi^{(1)}_1
 - \tfrac{7h^2}{20} \tfrac{d^2\varphi^{(1)}_1}{dy^2}  - \tfrac{h^3}{4}\tfrac{\partial^2 \psi^{(1)} }{\partial y^2}\\
 &\qquad   - h^4(\tfrac{133 }{360 }  \tfrac{d^4\varphi^{(1)}_1}{dy^4}+ \tfrac{3}{10}\tfrac{\partial^3 \psi^{(1)} }{\partial x\partial y^2})-h^5 (  \tfrac{5}{24 }\tfrac{\partial^4 \psi^{(1)} }{\partial x^2\partial y^2}- \tfrac{1}{16 }\tfrac{\partial^4 \psi^{(1)} }{\partial y^4} ),\\
\varphi^{(2)}_{1,h}
&:= \varphi^{(2)}_1    - \tfrac{3 h}{2} \psi^{(2)} + h^2( \tfrac{1}{5  }  \tfrac{d^2 \varphi^{(2)}_1}{dy^2}  - \tfrac{5}{4  }\tfrac{\partial \psi^{(2)} }{\partial x}) - h^3(\tfrac{3}{4}\tfrac{\partial^2 \psi^{(2)} }{\partial x^2}+ \tfrac{3}{10}\tfrac{\partial^2 \psi^{(2)} }{\partial y^2} ) \\
	&  \qquad   - h^4(  \tfrac{1}{360}\tfrac{d^4  \varphi^{(2)}_1}{dy^4}   +\tfrac{43}{120}\tfrac{\partial^3 \psi^{(2)} }{\partial x^3}+\tfrac{31}{120}\tfrac{\partial^3 \psi^{(2)} }{\partial x \partial y^2})-h^5(    \tfrac{23}{160}\tfrac{\partial^4 \psi^{(2)} }{\partial x^4}+ \tfrac{13}{80}\tfrac{\partial^4 \psi^{(2)} }{\partial x^2\partial y^2}+\tfrac{7}{480}\tfrac{\partial^4 \psi^{(2)} }{\partial y^4} ),\\
\varphi^{(2)}_{2,h}
&:=	 \varphi^{(2)}_2
 - \tfrac{7h^2}{20} \tfrac{d^2\varphi^{(2)}_2}{dx^2}  - \tfrac{h^3}{4}\tfrac{\partial^2 \psi^{(2)} }{\partial x^2}\\
 &\qquad   - h^4(\tfrac{133 }{360 }  \tfrac{d^4\varphi^{(2)}_2}{dx^4}+ \tfrac{3}{10}\tfrac{\partial^3 \psi^{(2)} }{\partial x^2\partial y})+h^5 ( \tfrac{1}{16 }\tfrac{\partial^4 \psi^{(2)} }{\partial x^4} - \tfrac{5}{24 }\tfrac{\partial^4 \psi^{(2)} }{\partial x^2\partial y^2}),\\
\varphi^{(1)}_{2,h}
&:= \varphi^{(1)}_2    - \tfrac{3 h}{2} \psi^{(1)} + h^2( \tfrac{1}{5  }  \tfrac{d^2 \varphi^{(1)}_2}{dx^2}  - \tfrac{5}{4  }\tfrac{\partial \psi^{(1)} }{\partial y}) - h^3( \tfrac{3}{10}\tfrac{\partial^2 \psi^{(1)} }{\partial x^2} +\tfrac{3}{4}\tfrac{\partial^2 \psi^{(1)} }{\partial y^2}) \\
	&  \qquad   - h^4(  \tfrac{1}{360}\tfrac{d^4  \varphi^{(1)}_2}{dx^4}   +\tfrac{31}{120}\tfrac{\partial^3 \psi^{(1)} }{\partial x^2\partial y }+\tfrac{43}{120}\tfrac{\partial^3 \psi^{(1)} }{\partial y^3})-h^5(  \tfrac{7}{480}\tfrac{\partial^4 \psi^{(1)} }{\partial x^4} + \tfrac{13}{80}\tfrac{\partial^4 \psi^{(1)} }{\partial x^2\partial y^2} + \tfrac{23}{160}\tfrac{\partial^4 \psi^{(1)} }{\partial y^4}) ,
\end{split}
\end{equation}	
where $\bpsi$ and $\varphi^{(r)}_j$ are defined in \eqref{RHS:2d}.

Recall that $\Gamma_{m,h}^x$ and $\Gamma_{n,h}^y$ in \eqref{Gamma:bd} denote  the set of nodes
in a neighborhood of $\Gamma_{m}^x$ and $\Gamma_n^y$ respectively (see green points  in Figure \ref{FDM:corner:special}). 
A 15-point FD operator at
boundary side points is given in the following lemma.
\begin{lemma}\label{theorem:side:1:generalized}
Define the following 15-point FD operators ($r=1,2$)
\begin{equation}
\label{scheme:u1:side:1}
\begin{split}
\mathcal{L}^{x,r}_h (u_h^{(r)})_{i,j} :=&\sum_{k=0}^2\sum_{\ell=-2}^2C^{(r)}_{k,\ell} (u_h^{(r)})_{i+\sigma_m^x k,j+\ell},\qquad \text{on }\Gamma_{m,h}^x,\\
\mathcal{L}^{y,r}_h (u_h^{(r)})_{i,j} := &\sum_{k=-2}^2\sum_{\ell=0}^2C^{(r')}_{\ell,k} (u_h^{(r)})_{i+k,j+\sigma_n^y\ell},\qquad \text{on }\Gamma_{n,h}^y,
\end{split}
\end{equation}
where $r'=1$ if $r=2$ and $r'=2$ if $r=1$, and $\sigma^x_m,\sigma^y_n\in \{-1,1\}$ such that $\alpha_m+2\sigma^x_m h,\beta_n+2\sigma^y_n h\in \overline\Omega$. Then there holds  $\left| \frac1{h^3} \mathcal{L}^{x,r}_h (u^{(r)})_{i,j}-\varphi_{1,\sigma_m^x h}^{(r)}\right| = \bo(h^6)$
at $({\textup\xa}_i,{\textup\ya}_j) \in \Gamma_{m,h}^x$ and $\left| \frac1{h^3} \mathcal{L}^{y,r}_h (u^{(r)})_{i,j}-\varphi_{2,\sigma_n^y h}^{(r)}\right| = \bo(h^6)$ at $({\textup\xa}_i,{\textup\ya}_j) \in \Gamma_{n,h}^y$ provided that $\bu$ is smooth, and where $\varphi_{1,\sigma_m^x h}^{(r)}$ and $\varphi_{2,\sigma_n^y h}^{(r)}$ are evaluated at $(\alpha_m, {\textup\ya}_j)$ and $( {\textup\xa}_i,\beta_n)$ respectively.
\end{lemma}

\begin{proof}
We only prove the first assertion with $r=1$, as the other cases are derived using similar techniques.
Fix $m\in \{1,2,\ldots,N_x\}$ and without loss of generality, assume $\sigma_m^x = 1$.
	Using the Taylor approximation at the base point $(\alpha_m,\ya_j)$, we have
	\be \label{u:original:bd}
	\bu(\alpha_m+x,\ya_j+y)
	=
	\sum_{(p,n)\in \ind_{M+1}}
	 \frac{\partial^{p+n}\bu(\alpha_m,\ya_j) }{\partial x^p \partial y^n}\frac{x^py^n}{p!n!}+\bo(h^{M+2}),
	\ee
	where  $x\in \{h, 2h,3h\}$ and $y\in \{0,\pm h, \pm 2h\}$ and $\ind_{M+1}$ is defined by \eqref{ind}.
Similar to the previous proof, we omit the subscripts for readability
and  use the following notations:
\begin{align*}
		\frac{\partial^{p+n}\bu }{\partial x^m \partial y^n}:=\frac{\partial^{p+n}\bu(\alpha_m,\ya_j) }{\partial x^p \partial y^n},\qquad  \frac{\partial^{p+n} {\bm \psi} }{\partial x^p \partial y^n}:=\frac{\partial^{p+n} {\bm \psi}(\alpha_m,\ya_j) }{\partial x^p \partial y^n}, \qquad  \frac{d^n \varphi^{(1)} }{dy^n}:=\frac{d^n \varphi^{(1)}(\alpha_m,\ya_j) }{dy^n}.
\end{align*}

By \eqref{Ckl:1:bd}, \eqref{scheme:u1:side:1},   and \eqref{u:original:bd} with $r=1$ and $M=7$,
\begin{align*}
\mathcal{L}_h^{x,1}(u^{(1)})_{i,j}
		&=   \tfrac{11h^2}{6} \tfrac{\partial^{2}u^{(1)} }{\partial y^2}+h^3\tfrac{\partial^{3}u^{(1)} }{\partial x\partial y^2}
		-\tfrac{25h^4}{24} \tfrac{\partial^{4}u^{(1)} }{\partial y^4}-\tfrac{7 h^5}{20} \tfrac{\partial^{5}u^{(1)} }{ \partial x \partial y^4}
		 +h^6\Big(\tfrac1{4}\tfrac{\partial^{6}u^{(1)} }{\partial x^4\partial y^2}\notag \\
		 &\qquad+\tfrac{1}{2}\tfrac{\partial^{6}u^{(1)} }{\partial x^2\partial y^4}-\tfrac{419}{2160}\tfrac{\partial^{6}u^{(1)} }{\partial y^6}\Big)
		 +h^7\Big(\tfrac{3}{10}\tfrac{\partial^{7}u^{(1)} }{\partial x^5\partial y^2}+\tfrac3{5}\tfrac{\partial^{7}u^{(1)} }{\partial x^3\partial y^4}-\tfrac{5}{72}\tfrac{\partial^{7}u^{(1)} }{ \partial x \partial y^6}\Big) \\
		 &\qquad+h^8\Big(\tfrac{5}{24}\tfrac{\partial^{8}u^{(1)} }{\partial x^6\partial y^2}+\tfrac{17}{48}\tfrac{\partial^{8}u^{(1)} }{\partial x^4\partial y^4}
		 +\tfrac{1}{12}\tfrac{\partial^{8}u^{(1)} }{\partial x^2\partial y^6}-\tfrac{359}{24192}\tfrac{\partial^{8}u^{(1)} }{ \partial y^8}\Big)+\bo(h^{9}).\notag
\end{align*}	
By \eqref{eqn:FourthOrderPDE},
\begin{equation*}
\tfrac{\partial^n u^{(1)}}{\partial y^n}=0, \qquad n\ge0;\qquad
\tfrac{\partial^{n+1} u^{(1)}}{\partial x\partial y^n}=\tfrac{d^{n-2} \varphi^{(1)}_1}{d y^{n-2}}, \qquad n\ge 2.
\end{equation*}
Consequently, there holds	
\begin{align}	 \label{bd:taylor:plug}
\mathcal{L}_h^{x,1}(u^{(1)})_{i,j}
&=  h^3\varphi^{(1)}_1
- \tfrac{7h^5}{20}\tfrac{d^{2}\varphi^{(1)}_1 }{d y^2}+h^6\Big(\tfrac{1}{4}\tfrac{\partial^{6}u^{(1)} }{\partial x^4\partial y^2}
+\tfrac{1}{2} \tfrac{\partial^{6}u^{(1)} }{\partial x^2\partial y^4}\Big) \notag \\
&\quad +h^7\Big(\tfrac{3}{10}\tfrac{\partial^{7}u^{(1)} }{\partial x^5\partial y^2}+\tfrac{3}{5}\tfrac{\partial^{7}u^{(1)} }{\partial x^3\partial y^4}-\tfrac{5}{72}\tfrac{d^{4}\varphi^{(1)}_1 }{d y^4}\Big) \\
&\qquad +h^8\Big(\tfrac{5}{24}\tfrac{\partial^{8}u^{(1)} }{\partial x^6\partial y^2}+\tfrac{17}{48}\tfrac{\partial^{8}u^{(1)} }{\partial x^4\partial y^4}+\tfrac{1}{12}\tfrac{\partial^{8}u^{(1)} }{\partial x^2\partial y^6}\Big)+\bo(h^{9}).\notag
	\end{align}

We then write
\begin{equation}\label{eqn:T1}
\begin{split}
\Big(\tfrac{1}{4}\tfrac{\partial^{6}u^{(1)} }{\partial x^4\partial y^2}
		+\tfrac{1}{2} \tfrac{\partial^{6}u^{(1)} }{\partial x^2\partial y^4}\Big)
		& = \tfrac1{4} \tfrac{\p^2}{\p y^2} \Delta^2 u^{(1)} -\tfrac{1}{4}\tfrac{\p^6 u^{(1)}}{\p y^6} =  -\tfrac1{4} \tfrac{\p^2}{\p y^2}  \psi^{(1)},
\end{split}
\end{equation}
and
\begin{equation}\label{eqn:T2}
\begin{split}
\Big(\tfrac{3}{10}\tfrac{\partial^{7}u^{(1)} }{\partial x^5\partial y^2}+\tfrac{3}{5}\tfrac{\partial^{7}u^{(1)} }{\partial x^3\partial y^4}-\tfrac{5}{72}\tfrac{d^{4}\varphi^{(1)}_1 }{ d y^4}\Big)
& = \tfrac{3}{10} \tfrac{\p^3}{\p x \p y^2} \Delta^2 u^{(1)} - \tfrac3{10} \tfrac{\p^7 u^{(1)}}{\p x \p y^6} -\tfrac{5}{72}\tfrac{d^{4}\varphi^{(1)}_1 }{ d y^4}\\
& = -\tfrac{3}{10} \tfrac{\p^3}{\p x \p y^2} \psi^{(1)}-\tfrac{133}{360} \tfrac{d^{4}\varphi^{(1)}_1 }{ d y^4}.
\end{split}
\end{equation}

Moreover, a simple calculation shows (recalling $\Delta^2 \bu =-\bpsi$)
\begin{align*}
\tfrac{-5}{24} \tfrac{\p^4 \psi^{(1)}}{\p x^2 \p y^2} + \tfrac1{16} \tfrac{\p^4\psi^{(1)}}{\p y^4} &= \tfrac{5}{24}\tfrac{\p^4}{\p x^2 \p y^2} \Delta^2 u^{(1)}-\tfrac1{16} \tfrac{\p^4}{ \p y^4} \Delta^2 u^{(1)}   \\
& = \tfrac5{24} \tfrac{\p^8 u^{(1)}}{\p x^6 \p y^2} + \tfrac{17}{48}\tfrac{\p^8 u^{(1)}}{\p x^4 \p^4 y} + \tfrac1{12} \tfrac{\p^8 u^{(1)}}{\p x^2\p y^6}- \tfrac1{16} \tfrac{\p^8 u^{(1)}}{\p y^8}.
\end{align*}
Consequently,
\begin{equation}
\label{eqn:T3}
\begin{split}
\tfrac{5}{24}\tfrac{\partial^{8}u^{(1)} }{\partial x^6\partial y^2}+\tfrac{17}{48}\tfrac{\partial^{8}u^{(1)} }{\partial x^4\partial y^4}+\tfrac{1}{12}\tfrac{\partial^{8}u^{(1)} }{\partial x^2\partial y^6} = \tfrac{-5}{24} \tfrac{\p^4 \psi^{(1)}}{\p x^2 \p y^2} + \tfrac1{16} \tfrac{\p^4\psi^{(1)}}{\p y^4}.
\end{split}
\end{equation}

Plugging in the identities \eqref{eqn:T1}--\eqref{eqn:T3} into \eqref{bd:taylor:plug} yields
		\[
		\begin{split}
\mathcal{L}_h^{x,1}(u^{(1)})_{i,j}
		& =  h^3\varphi^{(1)}_1
		-\tfrac{7 h^5}{20}\tfrac{d^{2}\varphi^{(1)}_1 }{d y^2}-\tfrac{h^6}{4} \tfrac{\p^2 \psi^{(1)}}{\p y^2}-h^7\Big( \tfrac{3}{10} \tfrac{\p^3\psi^{(1)}}{\p x \p y^2}	 +\tfrac{133}{360} \tfrac{d^{4}\varphi^{(1)}_1 }{ d y^4}\Big)\\
		& +h^8\Big(\tfrac{-5}{24} \tfrac{\p^4 \psi^{(1)}}{\p x^2 \p y^2} + \tfrac1{16} \tfrac{\p^4\psi^{(1)}}{\p y^4}\Big)+\bo(h^{9})\\
	& = h^3 \varphi_{1,h}^{(1)} + \bo(h^9),
	\end{split}
	\]
	where we used \eqref{rhos} in the last equality.  We then divide by $h^3$ to obtain
	the desired result.
\end{proof}
	\subsection{A sixth-order  FD operator for smooth $\bu$  at corner points}\label{FDMs:corners}
%
%
In this section, we propose a 2-point FD operator with the sixth-order of consistency in a neighborhood
of the set of corner points in $\Omega^c$ of \eqref{Omega:c} (cf.~blue points in Figure \ref{FDM:corner:special}).
%
%
To  describe the  sixth-order 2-point FD operator, we define
	 \begin{align}\label{right:corner}		 
	& \varphi^{(1)}_{{2,\rm left,bot},h}	 := \varphi^{(1)}_2  +h(\tfrac{d \varphi^{(1)}_2}{dx} -\tfrac{d \varphi^{(1)}_1}{dy}    -\tfrac{1}{2}\psi^{(1)}) \notag \\
	&\qquad \hspace{1.4cm} +h^2(\tfrac{2}{3}\tfrac{d^2 \varphi^{(1)}_2}{dx^2} -\tfrac{1}{2} \tfrac{d^2 \varphi^{(1)}_1}{dy^2}  -\tfrac{1}{2}\tfrac{\partial \psi^{(1)} }{\partial x}  +\tfrac{1}{6}\tfrac{\partial \psi^{(1)} }{\partial y} ) \notag\\
	&\qquad  \hspace{1.4cm} +\tfrac{h^3}{9 }(2\tfrac{d^3 \varphi^{(1)}_2}{dx^3}    -\tfrac{1}{2}\tfrac{d^3 \varphi^{(1)}_1}{dy^3}    -\tfrac{13}{4}\tfrac{\partial^2 \psi^{(1)} }{\partial x^2}-\tfrac{1}{2}\tfrac{\partial^2 \psi^{(1)} }{\partial x \partial y} +\tfrac{5}{4}\tfrac{\partial^2 \psi^{(1)} }{\partial y^2}) \notag\\
	&\qquad  \hspace{1.4cm} +h^4(\tfrac{13}{360}\tfrac{d^4 \varphi^{(1)}_2}{dx^4} +\tfrac{1}{36}\tfrac{d^4 \varphi^{(1)}_1}{dy^4}   -\tfrac{5}{36}\tfrac{\partial^3 \psi^{(1)} }{\partial x^3}-\tfrac{1}{10}\tfrac{\partial^3 \psi^{(1)} }{\partial x^2 \partial y}+\tfrac{1}{36}\tfrac{\partial^3 \psi^{(1)} }{\partial x \partial y^2}+ \tfrac{1}{24}\tfrac{\partial^3 \psi^{(1)} }{\partial y^3 }  )   \notag \\
	&\qquad  \hspace{1.4cm} + \tfrac{h^5}{60}(\tfrac{d^5 \varphi^{(1)}_1}{dy^5}   - \tfrac{59}{24}\tfrac{\partial^4 \psi^{(1)} }{\partial x^4} - 3\tfrac{\partial^4 \psi^{(1)} }{\partial x^3 \partial y} - \tfrac{1}{4}\tfrac{\partial^4 \psi^{(1)} }{\partial x^2 \partial y^2} + \tfrac{\partial^4 \psi^{(1)} }{\partial x \partial y^3} +\tfrac{1}{24}\tfrac{\partial^4 \psi^{(1)} }{\partial y^4}), \notag \\
	&  \varphi^{(2)}_{{2,\rm left,bot},h}	 :=	  \varphi^{(2)}_2    +h(\tfrac{d \varphi^{(2)}_2}{dx} -\tfrac{1}{4}\psi^{(2)}) \notag \\
	&\qquad  \hspace{1.4cm}  + h^2(\tfrac{1}{2} \tfrac{d^2 \varphi^{(2)}_2}{dx^2} -\tfrac{1}{6}\tfrac{d^2 \varphi^{(2)}_1}{dy^2}    -\tfrac{1}{3}\tfrac{\partial \psi^{(2)} }{\partial x} -\tfrac{1}{12}\tfrac{\partial \psi^{(2)} }{\partial y}) \notag \\
	& \qquad  \hspace{1.4cm} + \tfrac{h^3}{36 }(5\tfrac{d^3 \varphi^{(2)}_2}{dx^3} -2 \tfrac{d^3 \varphi^{(2)}_1}{dy^3}  -\tfrac{13}{2}\tfrac{\partial^2 \psi^{(2)} }{\partial x^2}-4\tfrac{\partial^2 \psi^{(2)} }{\partial x \partial y} +\tfrac{5}{2}\tfrac{\partial^2 \psi^{(2)} }{\partial y^2}) \notag \\
	& \qquad  \hspace{1.4cm} + \tfrac{h^4}{12} (\tfrac{13}{45} \tfrac{d^4 \varphi^{(2)}_2}{dx^4} +\tfrac{1}{8} \tfrac{d^4 \varphi^{(2)}_1}{dy^4} -\tfrac{7}{8}\tfrac{\partial^3 \psi^{(2)} }{\partial x^3 }-\tfrac{67}{90}\tfrac{\partial^3 \psi^{(2)} }{\partial x^2 \partial y} +\tfrac{1}{4} \tfrac{\partial^3 \psi^{(2)} }{\partial x \partial y^2}    +\tfrac{29}{90}\tfrac{\partial^3 \psi^{(2)} }{\partial y^3}) \notag \\
	& \qquad  \hspace{1.4cm} -\tfrac{h^5}{80}(    \tfrac{1}{6}\tfrac{d^5 \varphi^{(2)}_2}{dx^5}  - \tfrac{1}{2}\tfrac{d^5 \varphi^{(2)}_1}{dy^5}   +\tfrac{59}{36}\tfrac{\partial^4 \psi^{(2)} }{\partial x^4}  +\tfrac{13}{6}\tfrac{\partial^4 \psi^{(2)} }{\partial x^3 \partial y}  +\tfrac{1}{6}\tfrac{\partial^4 \psi^{(2)} }{\partial x^2 \partial y^2} -\tfrac{\partial^4 \psi^{(2)} }{\partial x \partial y^3}  - \tfrac{1}{36} \tfrac{\partial^4 \psi^{(2)} }{\partial y^4} ), \notag \\
	&	\varphi^{(1)}_{{2,\rm left,up},h}	 := \varphi^{(1)}_2  +h(\tfrac{d \varphi^{(1)}_2}{dx} -\tfrac{d \varphi^{(1)}_1}{dy}    +\tfrac{1}{2}\psi^{(1)})  \\
	&\qquad  \hspace{1.4cm} +h^2(\tfrac{2}{3}\tfrac{d^2 \varphi^{(1)}_2}{dx^2} +\tfrac{1}{2} \tfrac{d^2 \varphi^{(1)}_1}{dy^2}  +\tfrac{1}{2}\tfrac{\partial \psi^{(1)} }{\partial x}  +\tfrac{1}{6}\tfrac{\partial \psi^{(1)} }{\partial y} ) \notag \\
	&\qquad  \hspace{1.4cm} +\tfrac{h^3}{9 }(2\tfrac{d^3 \varphi^{(1)}_2}{dx^3}    -\tfrac{1}{2}\tfrac{d^3 \varphi^{(1)}_1}{dy^3}   +\tfrac{13}{4}\tfrac{\partial^2 \psi^{(1)} }{\partial x^2} -\tfrac{1}{2}\tfrac{\partial^2 \psi^{(1)} }{\partial x \partial y} -\tfrac{5}{4}\tfrac{\partial^2 \psi^{(1)} }{\partial y^2}) \notag \\
	&\qquad  \hspace{1.4cm}  +h^4(\tfrac{13}{360}\tfrac{d^4 \varphi^{(1)}_2}{dx^4} -\tfrac{1}{36}\tfrac{d^4 \varphi^{(1)}_1}{dy^4}  +\tfrac{5}{36}\tfrac{\partial^3 \psi^{(1)} }{\partial x^3}-\tfrac{1}{10}\tfrac{\partial^3 \psi^{(1)} }{\partial x^2 \partial y}-\tfrac{1}{36}\tfrac{\partial^3 \psi^{(1)} }{\partial x \partial y^2}+ \tfrac{1}{24}\tfrac{\partial^3 \psi^{(1)} }{\partial y^3 })     \notag \\
	&\qquad \hspace{1.4cm} + \tfrac{h^5}{60}(\tfrac{d^5 \varphi^{(1)}_1}{dy^5}   + \tfrac{59}{24}\tfrac{\partial^4 \psi^{(1)} }{\partial x^4} - 3\tfrac{\partial^4 \psi^{(1)} }{\partial x^3 \partial y} + \tfrac{1}{4}\tfrac{\partial^4 \psi^{(1)} }{\partial x^2 \partial y^2} + \tfrac{\partial^4 \psi^{(1)} }{\partial x \partial y^3} -\tfrac{1}{24}\tfrac{\partial^4 \psi^{(1)} }{\partial y^4}), \notag \\
	&	\varphi^{(2)}_{{2,\rm left,up},h}	 :=	  \varphi^{(2)}_2   +h(\tfrac{d \varphi^{(2)}_2}{dx} +\tfrac{1}{4}\psi^{(2)}) \notag \\
	&\qquad  \hspace{1.4cm} + h^2(\tfrac{1}{2} \tfrac{d^2 \varphi^{(2)}_2}{dx^2} +\tfrac{1}{6}\tfrac{d^2 \varphi^{(2)}_1}{dy^2}  +\tfrac{1}{3}\tfrac{\partial \psi^{(2)} }{\partial x}   -\tfrac{1}{12}\tfrac{\partial \psi^{(2)} }{\partial y}) \notag \\
	& \qquad  \hspace{1.4cm} + \tfrac{h^3}{36 }(5\tfrac{d^3 \varphi^{(2)}_2}{dx^3} -2 \tfrac{d^3 \varphi^{(2)}_1}{dy^3} +\tfrac{13}{2}\tfrac{\partial^2 \psi^{(2)} }{\partial x^2}-4\tfrac{\partial^2 \psi^{(2)} }{\partial x \partial y} -\tfrac{5}{2}\tfrac{\partial^2 \psi^{(2)} }{\partial y^2} ) \notag \\
	& \qquad  \hspace{1.4cm} + \tfrac{h^4}{12} (\tfrac{13}{45} \tfrac{d^4 \varphi^{(2)}_2}{dx^4} -\tfrac{1}{8} \tfrac{d^4 \varphi^{(2)}_1}{dy^4} +\tfrac{7}{8}\tfrac{\partial^3 \psi^{(2)} }{\partial x^3 } -\tfrac{67}{90}\tfrac{\partial^3 \psi^{(2)} }{\partial x^2 \partial y}-\tfrac{1}{4} \tfrac{\partial^3 \psi^{(2)} }{\partial x \partial y^2}    +\tfrac{29}{90}\tfrac{\partial^3 \psi^{(2)} }{\partial y^3}) \notag \\
	& \qquad  \hspace{1.4cm} -\tfrac{h^5}{80}(   \tfrac{1}{6}\tfrac{d^5 \varphi^{(2)}_2}{dx^5}  - \tfrac{1}{2}\tfrac{d^5 \varphi^{(2)}_1}{dy^5}    -\tfrac{59}{36}\tfrac{\partial^4 \psi^{(2)} }{\partial x^4} +\tfrac{13}{6}\tfrac{\partial^4 \psi^{(2)} }{\partial x^3 \partial y}   -\tfrac{1}{6}\tfrac{\partial^4 \psi^{(2)} }{\partial x^2 \partial y^2} -\tfrac{\partial^4 \psi^{(2)} }{\partial x \partial y^3}+ \tfrac{1}{36} \tfrac{\partial^4 \psi^{(2)} }{\partial y^4}  ), \notag
\end{align}
where $\bpsi$ and $\varphi^{(r)}_j$ are defined in \eqref{RHS:2d}.
%
%
%
%
%
We provide 2-point FD operators at $(\xa_i\pm h,\ya_j\pm h),(\xa_i\pm h,\ya_j\mp h)\in \Omega^c$  (see blue points in Figure \ref{FDM:corner:special}) in the following lemma.
\begin{lemma}\label{theorem:corner:1}
	Define the following 2-point FD operators (r=1,2)
	\begin{align}	
		\mathcal{L}^{{\rm left,bot},r}_h (u_h^{(r)})_{i,j} &:= \lambda_r (u_h^{(r)})_{i,j}-\tfrac{\lambda_r}{2}(u_h^{(r)})_{i+1,j},\\
	\mathcal{L}^{{\rm left,up},r}_h (u_h^{(r)})_{i,j} &:= \tfrac{\lambda_r}{2}(u_h^{(r)})_{i+1,j}-\lambda_r (u_h^{(r)})_{i,j},\\
	\mathcal{L}^{{\rm right,bot},r}_h (u_h^{(r)})_{i,j} 	 &:= \lambda_r (u_h^{(r)})_{i,j}-\tfrac{\lambda_r}{2}(u_h^{(r)})_{i-1,j},\\
\mathcal{L}^{{\rm right,up},r}_h (u_h^{(r)})_{i,j} &:= \tfrac{\lambda_r}{2}(u_h^{(r)})_{i-1,j}-\lambda_r (u_h^{(r)})_{i,j},
	\end{align}
	where $\lambda_r=-4$ if $r=1$ and $\lambda_r=-2$ if $r=2$.

If $\bu$ is smooth,	then for $r\in \{1,2\}$, there holds
	\begin{itemize}
		\item[(i)] 	$\left| \frac1{h^3} \mathcal{L}^{{\rm left, bot},r}_h (u^{(r)})_{i,j}-\varphi_{2,{\rm left, bot},h}^{(r)}\right| = \bo(h^6)$
		at $({\textup\xa}_i,{\textup\ya}_j)$ with $ ({\textup\xa}_i-h,{\textup\ya}_j-h)\in \Omega^c$, where $\varphi_{2,{\rm left, bot},h}^{(r)}$ is evaluated at $ ({\textup\xa}_i-h,{\textup\ya}_j-h)$;
		\item[(ii)]  $\left| \frac1{h^3} \mathcal{L}^{{\rm left, up},r}_h (u^{(r)})_{i,j}-\varphi_{2,{\rm left, up},h}^{(r)}\right| = \bo(h^6)$ at
		 $({\textup\xa}_i,{\textup\ya}_j)$ with $ ({\textup\xa}_i-h,{\textup\ya}_j+h)\in \Omega^c$, where $\varphi_{2,{\rm left, up},h}^{(r)}$ is evaluated at $ ({\textup\xa}_i-h,{\textup\ya}_j+h)$;
		\item[(iii)]  $\left| \frac1{h^3} \mathcal{L}^{{\rm right, bot},r}_h (u^{(r)})_{i,j}-\varphi_{2,{\rm left, up},-h}^{(r)}\right| = \bo(h^6)$
		at $({\textup\xa}_i,{\textup\ya}_j)$ with $ ({\textup\xa}_i+h,{\textup\ya}_j-h)\in \Omega^c$, where $\varphi_{2,{\rm left, up},-h}^{(r)}$ is evaluated at $ ({\textup\xa}_i+h,{\textup\ya}_j-h)$;
		\item[(iv)]  $\left| \frac1{h^3} \mathcal{L}^{{\rm right, up},r}_h (u^{(r)})_{i,j}-\varphi_{2,{\rm left, bot},-h}^{(r)}\right| = \bo(h^6)$
		at $({\textup\xa}_i,{\textup\ya}_j)$ with $ ({\textup\xa}_i+h,{\textup\ya}_j+h)\in \Omega^c$,
		where $\varphi_{2,{\rm left, bot},-h}^{(r)}$ is evaluated at $ ({\textup\xa}_i+h,{\textup\ya}_j+h)$.
	\end{itemize}
	\end{lemma}
\begin{proof}
The proof is similar to the proof of \cref{theorem:side:1:generalized} and is therefore omitted.
\end{proof}
	\subsection{A 1-point  FD operator for smooth $\bu$  at special grid points}\label{FDMs:special:point}
If $\Omega$ is a concave domain (see Figure \ref{FDM:corner:special} for an example), it is straightforward to obtain an arbitrary high-order 1-point FD operator for some special grid points in a neighborhood
of a corner by using the Dirichlet boundary condition and  the Taylor approximations:
\begin{equation}
\label{eqn:SimpleTaylor}
	\begin{split}
 \bu(\xa_i\pm h,\ya_j)&=\sum_{k=0}^M\tfrac{\p^k {\bu (\xa_i,\ya_j)}}{k!\p x^k} (\pm h)^k+\bo(h^{M+1}), \\ 	
 \bu(\xa_i,\ya_j\pm h)&=\sum_{k=0}^M\tfrac{\p^k {\bu (\xa_i,\ya_j)}}{k!\p y^k} (\pm h)^k+\bo(h^{M+1}), \qquad \bu  = \bm 0, \qquad \text{on } \partial \Omega,
\end{split}
\end{equation}
where $M\in \NN$.
%
%

Suppose the grid point $(\xa_i,\ya_j)$ satisfies  $ ({\textup\xa}_i\pm h,{\textup\ya}_j)\in \Omega^c$ 
and $({\textup\xa}_i\pm 2h,{\textup\ya}_j)\in \Gamma_{n}^y$ or  the grid point $(\xa_i,\ya_j)$ satisfies
$ ({\textup\xa}_i,{\textup\ya}_j\pm h)\in \Omega^c$ 
and $({\textup\xa}_i,{\textup\ya}_j\pm 2h)\in \Gamma_{m}^x$
(see red points in Figure \ref{FDM:corner:special}). Then $\left| \bu_{i,j} \right|=\bo(h^{M+1})$ at $(\xa_i,\ya_j)$ for any $M\in \NN$ by \eqref{eqn:SimpleTaylor}.
Thus, we specify that $(\bu_h)_{i,j}=0$ at such grid points in the FD scheme.
\subsection{The FDM for data with less regularity}\label{FDMs:singular}
In \cref{FDMs:interior,FDMs:sides,FDMs:corners,FDMs:special:point}, we derived sixth-order FD operators assuming
the data and $\bu$ are smooth. In the case $\bpsi$, or $\varphi_j^{(r)}$ has less regularity, some terms in the 
right-hand side of our proposed FDM may be undefined.  In this case, we simply drop these 
terms to approximate $\bu$ with less smoothness.
Precisely, for any  $\xi\in \Big\{\tfrac{\p^{m+n} {\bm \psi}}{\p  x^m \p  y^n}, \tfrac{d^n\varphi^{(r)}_1}{dy^n}, \tfrac{d^m\varphi^{(r)}_2}{dx^m}\Big\}$, if $\xi$ is not defined at the corresponding base point in \eqref{bpsi:h}, \eqref{rhos}, \eqref{right:corner}, then we set this $\xi=0$. We test the case $\bu$ is singular in \cref{example3} to gauge the performance of the method.
	\subsection{The matrix form of the proposed FDM}\label{Matrix:FDMs} We constructed sixth-order FD operators with decoupling 
	properties at interior grid points (see the black point in $\Omega$ in Figure \ref{FDM:corner:special}) in  \cref{theorem:interior}, at boundary side points (see the green point in Figure \ref{FDM:corner:special}) in  \cref{theorem:side:1:generalized}, at boundary corner points (see the blue point in Figure \ref{FDM:corner:special}) in  \cref{theorem:corner:1}, and at special gird points (see the red point in Figure \ref{FDM:corner:special}) in  \cref{FDMs:special:point}. 
Now, it is straightforward to derive the following diagonal block matrix form of our proposed FDM to decouple $u_h^{(1)}$ and $u_h^{(2)}$  at all grid points in the following theorem.

\begin{theorem}\label{thm:matrix:form}
The matrix form of the proposed FDM  at all  grid points can be expressed as
\be\label{matrix:A1:A2}
\begin{bmatrix}
	A_1 & \bf{0} \\
	\bf{0} & A_2
\end{bmatrix}  \begin{bmatrix}
u_h^{(1)}  \\
 u_h^{(2)}
\end{bmatrix} =\begin{bmatrix}
b_1  \\
b_2
\end{bmatrix}, \quad i.e., \quad A_ru_h^{(r)}=b_r, \quad\text{with} \quad r=1,2,
\ee
where $A_1$ and $A_2$ are  two constant matrices,  $b_1$ and $b_2$ depend on ${\bm f}$ and $\phi$ of \eqref{Model:Original}, but $b_1$ and $b_2$ are independent of both the pressure $p$ and the constant kinematic viscosity $\nu$  of \eqref{Model:Original}.
\end{theorem}
\subsection{Approximating the pressure gradient}\label{sec:pxpy}
Using the numerical velocity approximation $\bu_h$ computed by the FDM in Lemmas \ref{theorem:interior}--\ref{theorem:corner:1}, we  introduce a finite difference scheme to approximate the gradient of the pressure $p$. 
Without loss of generality, we assume that  $ a^y_n\le x\le b_n^y$ and  $a^x_m\le y\le b^x_m$ in this subsection.
By \eqref{Model:Original}, we have
\be\label{pxpy}
p_x  =f^{(1)}+\nu \Delta u^{(1)}, \qquad  \qquad
p_y  =f^{(2)}+\nu \Delta u^{(2)}.
\ee
For any smooth 1D function $a(x)\in C^7(\R)$,
	\begin{align*}
	 a_{xx}(\xa_i)=&\tfrac{1}{h^2}\big[\tfrac{1}{90}a(\xa_{i-3})-\tfrac{3}{20} a(\xa_{i-2})+  \tfrac{3}{2} a(\xa_{i-1})-  \tfrac{49}{18} a(\xa_{i})\\
	&+\tfrac{3}{2}  a(\xa_{i+1})-\tfrac{3}{20}  a(\xa_{i+2})+\tfrac{1}{90} a(\xa_{i+3})\big]+\bo(h^{6}), \qquad \text{if} \qquad a^y_n \le \xa_{i\pm 3} \le b^y_n,\\
	 a_{xx}(\xa_i)=&\tfrac{1}{h^2}\big[ -\tfrac{7}{10} a(\xa_{i\pm7}) +\tfrac{1019}{180}  a(\xa_{i\pm6}) -\tfrac{201}{10}  a(\xa_{i\pm5}) +41  a(\xa_{i\pm4})-  \tfrac{949}{18} a(\xa_{i\pm3})\\
	&+  \tfrac{879}{20} a(\xa_{i\pm2})   	 -\tfrac{223}{10} a(\xa_{i\pm1})+\tfrac{469}{90}a(\xa_{i}) \big]+\bo(h^{6}), \qquad \text{if} \qquad a^y_n \le \xa_{i\pm7} \text{ and } \xa_{i}\le b^y_n, \\
	 a_{xx}(\xa_i)=&\tfrac{1}{h^2}\big[  \tfrac{11}{180} a(\xa_{i\pm6}) -\tfrac{1}{2}  a(\xa_{i\pm5})+\tfrac{9}{5}  a(\xa_{i\pm4})-\tfrac{67}{18}  a(\xa_{i\pm3})+  \tfrac{19}{4} a(\xa_{i\pm2})\\
	&-  \tfrac{27}{10} a(\xa_{i\pm1})-\tfrac{7}{18} a(\xa_{i})+\tfrac{7}{10}a(\xa_{i\mp1})\big]+\bo(h^{6}), \qquad \text{if} \qquad a^y_n \le \xa_{i\pm6} \text{ and } \xa_{i\mp1}\le b^y_n, \\
	 a_{xx}(\xa_i)=&\tfrac{1}{h^2}\big[  -\tfrac{1}{90} a(\xa_{i\pm5})+\tfrac{4}{45}  a(\xa_{i\pm4}) -\tfrac{3}{10}  a(\xa_{i\pm3}) +\tfrac{17}{36}  a(\xa_{i\pm2})+  \tfrac{13}{18} a(\xa_{i\pm1})\\
	& -  \tfrac{21}{10} a(\xa_{i}) +\tfrac{107}{90} a(\xa_{i\mp1}) -\tfrac{11}{180}a(\xa_{i\mp2})\big]+\bo(h^{6}), \qquad \text{if} \qquad a^y_n \le \xa_{i\pm5} \text{ and } \xa_{i\mp2}\le b^y_n.
	\end{align*}
We can then use the above FD operators to compute $((\bu_{xx})_{h})_{i,j}$
at all grid points $(\xa_i,\ya_j)$. Similarly, we can compute  $((\bu_{yy})_{h})_{i,j}$.

 By \eqref{pxpy}, we approximate gradient $(p_x,p_y)$ of the pressure $p$ with the sixth-order accuracy   as follows:
	\begin{align*}
	&((p_x)_h)_{i,j}  =f^{(1)}(\xa_i,\ya_j)+\nu ((u^{(1)}_{xx})_{h})_{i,j}+\nu ((u^{(1)}_{yy})_{h})_{i,j}, \\
	&((p_y)_h)_{i,j}  =f^{(2)}(\xa_i,\ya_j)+\nu ((u^{(2)}_{xx})_{h})_{i,j}+\nu ((u^{(2)}_{yy})_{h})_{i,j}.
	\end{align*}

	\section{Numerical Experiments}\label{numerical:test}
In this section, we test the FDM for the Stokes problem \eqref{Model:Original},
measuring the error  in the  $\ell_\infty$-norm:
\begin{align*}
& \|\bu_h-\bu\|_{\infty}:=\max\left\{\left\|u_h^{(1)}-u^{(1)}\right\|_{\infty},\left\|u_h^{(2)}-u^{(2)}\right\|_{\infty}\right\},\\
& \|(\nab p)_h-\nab p\|_{\infty}:=\max\left\{\|(p_x)_h-p_x\|_{\infty},\|(p_y)_h-p_y\|_{\infty}\right\}.
\end{align*}
 In addition, we compute $\kappa$, the condition number of the matrix in \eqref{matrix:A1:A2}. We solve $\bu_h=(u^{(1)}_h,u^{(2)}_h)$ together in \cref{example1,example2}, and solve $u^{(1)}_h,u^{(2)}_h$ separately in \cref{example3}.
\begin{example}\label{example1}
	\normalfont
	Consider the square domain $\Omega=(-1,1)^2$, and
	let the functions in \eqref{Model:Original} be given by
	\begin{align*}
	 \bu=\left( \cos(3x-3y) \exp(y), \exp(x) \sin(3 x)\cos(3 y) \right ),\qquad p=\sin(x-3y),\qquad \nu=1,10^{-3},10^{-6}.
	\end{align*}
	The functions ${\bm f}, \phi,$ and ${\bm g}$ are then obtained by plugging the above functions into \eqref{Model:Original}.
	The numerical results are presented in Tables \ref{tab:example:1:u}--\ref{tab:example:1:p} and Figures \ref{fig:u:example:1}--\ref{fig:p:example:1}.  From \cref{thm:all:pde:2D,thm:matrix:form}, the viscosity $\nu$  and pressure $p$
	do not appear in the reformulated PDEs  \eqref{eqn:FourthOrderPDE} nor in  the linear system \eqref{matrix:A1:A2} of our proposed sixth-order FDM, i.e.,	our FDM is viscosity- and pressure-robust. In particular,  
	Tables \ref{tab:example:1:u}--\ref{tab:example:1:p} show that the error $\|\bu_h-\bu\|_{\infty}$ is independent of $\nu$ and $p$, and both
	$\|\bu_h-\bu\|_\infty$  and $ \|(\nab p)_h-\nab p\|_{\infty}$ converge with sixth order. Finally,
	we see that the condition number of the linear system  \eqref{matrix:A1:A2} scales like $\bo(h^{-4})$, which is expected since the scheme
	is a discretization of a fourth-order PDE.
\end{example}
\begin{table}[htbp]
	\caption{Performance in \cref{example1}  of the proposed FDM.}
	\centering
	\setlength{\tabcolsep}{0.7mm}{
		 \begin{tabular}{c|c|c|c|c|c|c|c|c}
			\hline
			\multicolumn{1}{c|}{}  &
			 \multicolumn{2}{c|}{$\nu=1$}  &
			 \multicolumn{2}{c|}{$\nu=10^{-3}$}  &
			 \multicolumn{2}{c|}{$\nu=10^{-6}$} &
			 \multicolumn{2}{c}{any $\nu>0$} 
			  \\
			\hline
			$h$
			&  $\|\bu_h-\bu\|_{\infty}$
			&order &  $\|\bu_h-\bu\|_{\infty}$ & order &  $\|\bu_h-\bu\|_{\infty}$
			&order & $\kappa$ & ratio of $\kappa$ \\
			\hline
$\tfrac{1}{2^2}$ & 2.3354E-01 &  & 2.3354E-01 &  & 2.3354E-01 &    
&1.74E+03 & \\
$\tfrac{1}{2^3}$ & 2.4130E-03 & 6.60 & 2.4130E-03 & 6.60 & 2.4130E-03 & 6.60  
& 4.30E+04 & 24.7\\
$\tfrac{1}{2^4}$ & 2.8524E-05 & 6.40 & 2.8524E-05 & 6.40 & 2.8524E-05 & 6.40  
& 6.77E+05 & 15.7\\
$\tfrac{1}{2^5}$ & 4.4981E-07 & 5.99 & 4.4981E-07 & 5.99 & 4.4981E-07 & 5.99   
& 1.04E+07 & 15.3\\
$\tfrac{1}{2^6}$ & 7.5936E-09 & 5.89 & 7.5936E-09 & 5.89 & 7.5936E-09 & 5.89  
& 1.63E+08 & 15.7\\		
			\hline
	\end{tabular}}
\label{tab:example:1:u}
\end{table}	
\begin{table}[htbp]
	\caption{Performance in \cref{example1}  of the proposed FDM.}
	\centering
	\setlength{\tabcolsep}{0.5mm}{
		 \begin{tabular}{c|c|c|c|c|c|c}
			\hline
			\multicolumn{1}{c|}{}  &
			 \multicolumn{2}{c|}{$\nu=1$}  &
			 \multicolumn{2}{c|}{$\nu=10^{-3}$}  &
			 \multicolumn{2}{c}{$\nu=10^{-6}$}   \\
			\hline
			$h$
			& $\|(\nab p)_h-\nab p\|_{\infty}$
			&order &  $\|(\nab p)_h-\nab p\|_{\infty}$ & order &  $\|(\nab p)_h-\nab p\|_{\infty}$\\
			\hline
$\tfrac{1}{2^2}$ & 5.4146E+0 &  & 5.4146E-03 &  & 5.4146E-06 & \\  
$\tfrac{1}{2^3}$ & 6.3722E-02 & 6.41 & 6.3722E-05 & 6.41 & 6.3722E-08 & 6.41\\ 
$\tfrac{1}{2^4}$ & 1.1137E-03 & 5.84 & 1.1137E-06 & 5.84 & 1.1137E-09 & 5.84\\ 
$\tfrac{1}{2^5}$ & 1.9109E-05 & 5.86 & 1.9109E-08 & 5.86 & 1.9109E-11 & 5.86\\ 
$\tfrac{1}{2^6}$ & 3.2051E-07 & 5.90 & 3.2051E-10 & 5.90 & 3.2051E-13 & 5.90\\ 
			\hline
	\end{tabular}}
\label{tab:example:1:p}
\end{table}	
\begin{figure}[htbp]
	\centering
	\begin{subfigure}[b]{0.23\textwidth}
		\includegraphics[width=4cm,height=4cm]{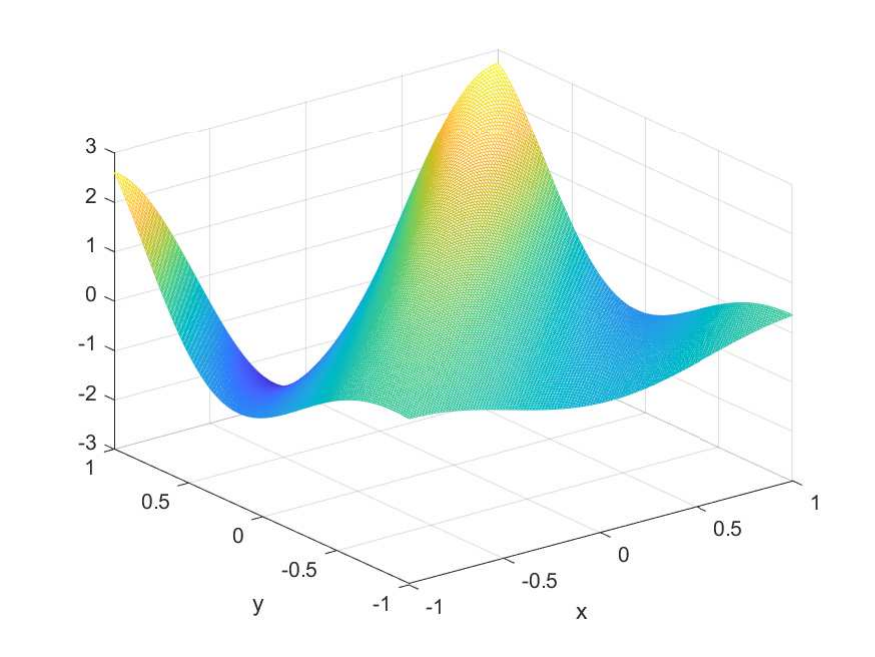}
	\end{subfigure}
	\begin{subfigure}[b]{0.23\textwidth}
		\includegraphics[width=4cm,height=4cm]{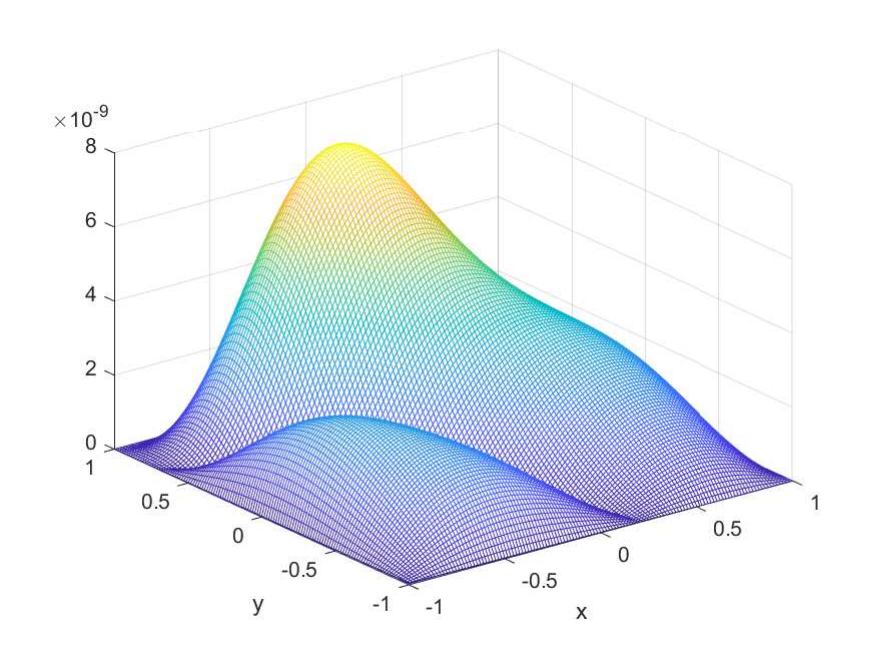}
	\end{subfigure}
	\begin{subfigure}[b]{0.23\textwidth}
		\includegraphics[width=4cm,height=4cm]{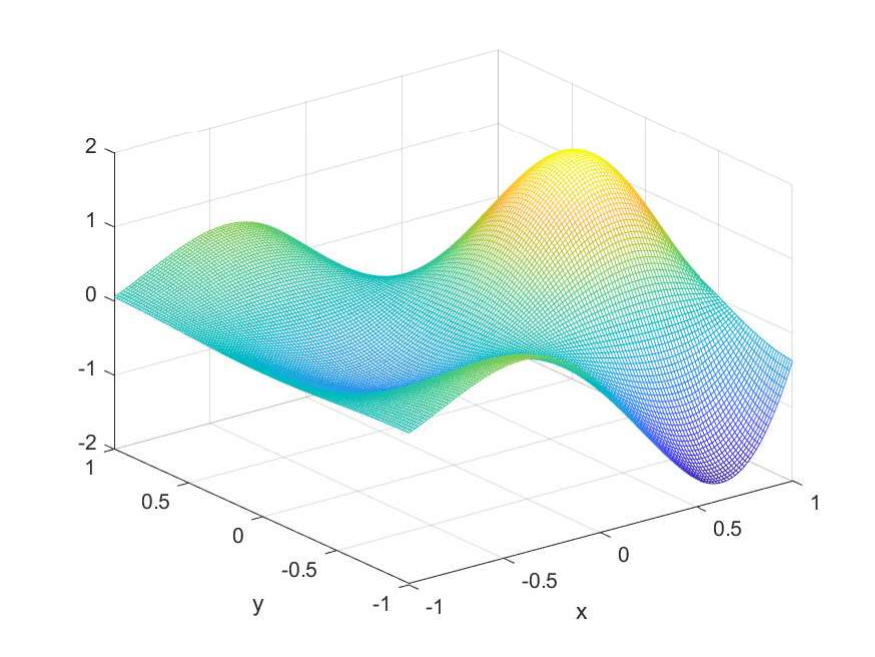}
	\end{subfigure}
	\begin{subfigure}[b]{0.23\textwidth}
		\includegraphics[width=4cm,height=4cm]{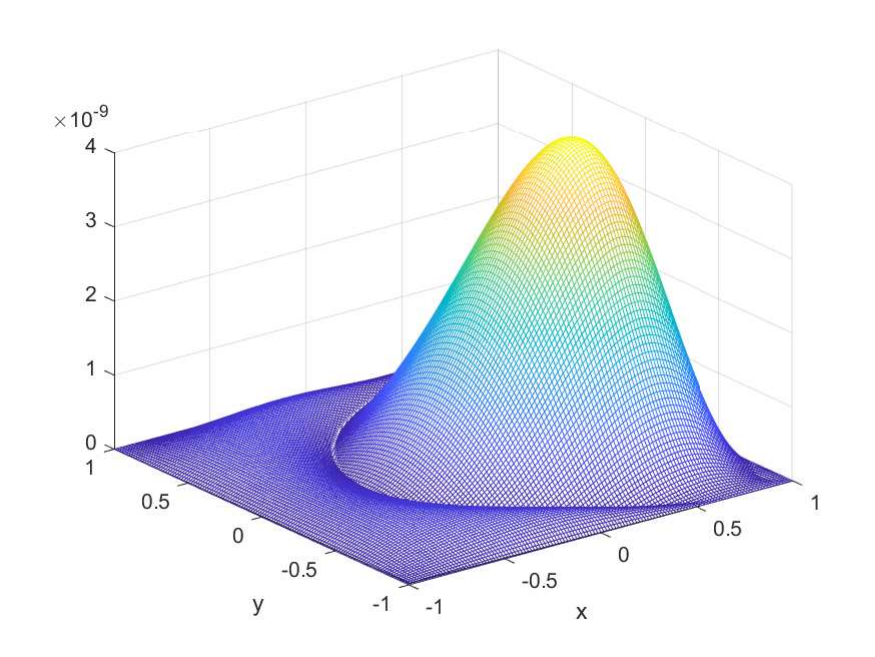}
	\end{subfigure}
	\caption{\cref{example1}:  $u^{(1)}$ (first panel), $|u_h^{(1)}-u^{(1)}|$ (second panel), $u^{(2)}$ (third panel), and $|u_h^{(2)}-u^{(2)}|$ (fourth panel) at all grid points in $\overline{\Omega}$ with $h=\tfrac{1}{2^6}$ and $\nu=10^{-6}$.}
	\label{fig:u:example:1}
\end{figure}	
\begin{figure}[htbp]
	\centering
	\begin{subfigure}[b]{0.23\textwidth}
		\includegraphics[width=4cm,height=4cm]{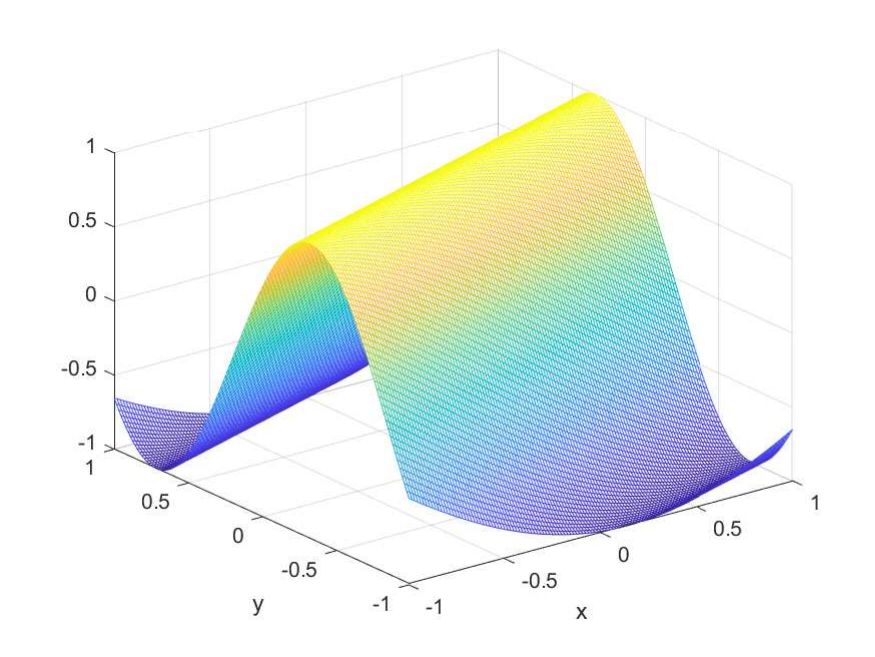}
	\end{subfigure}
	\begin{subfigure}[b]{0.23\textwidth}
		\includegraphics[width=4cm,height=4cm]{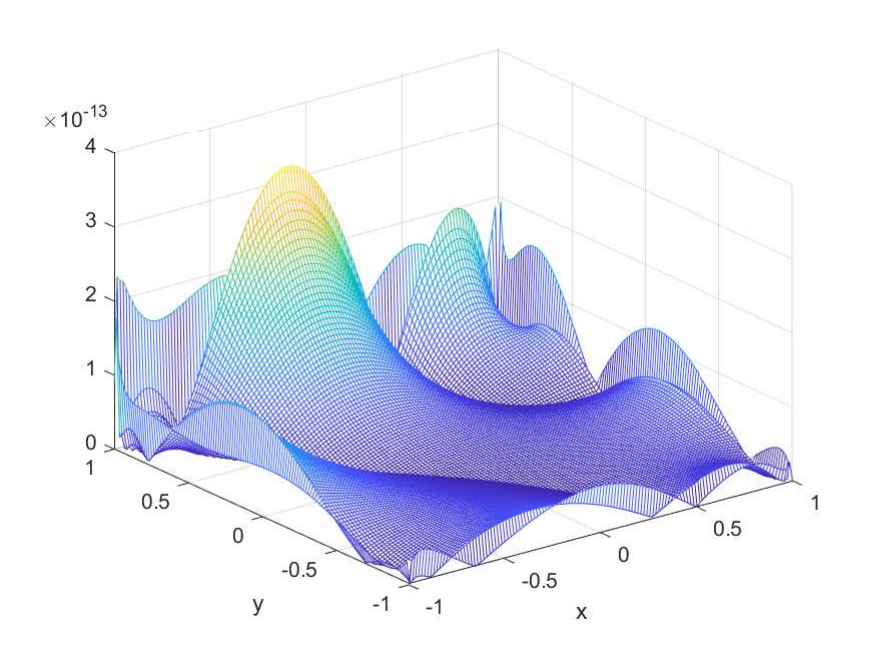}
	\end{subfigure}
	\begin{subfigure}[b]{0.23\textwidth}
		\includegraphics[width=4cm,height=4cm]{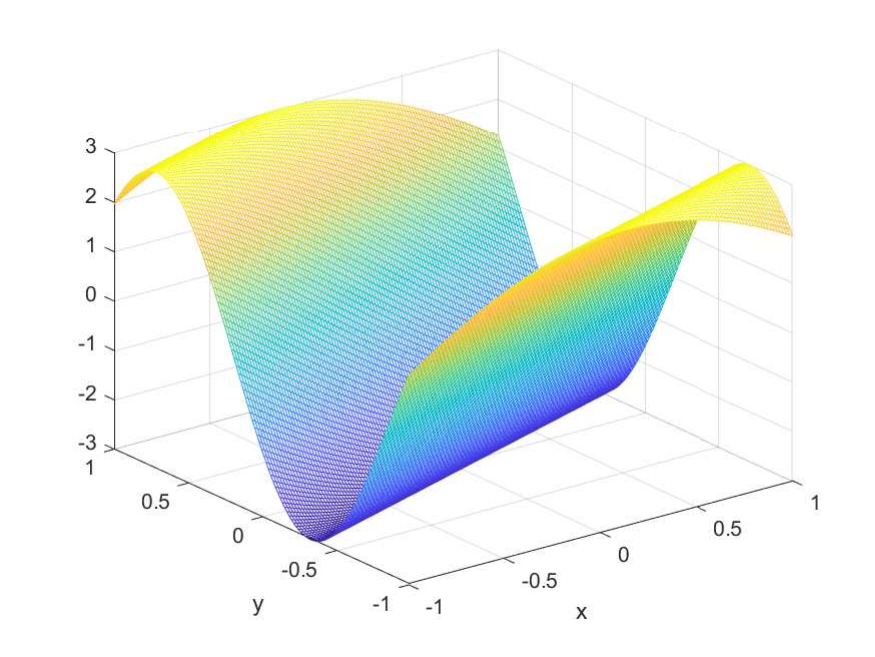}
	\end{subfigure}
	\begin{subfigure}[b]{0.23\textwidth}
		\includegraphics[width=4cm,height=4cm]{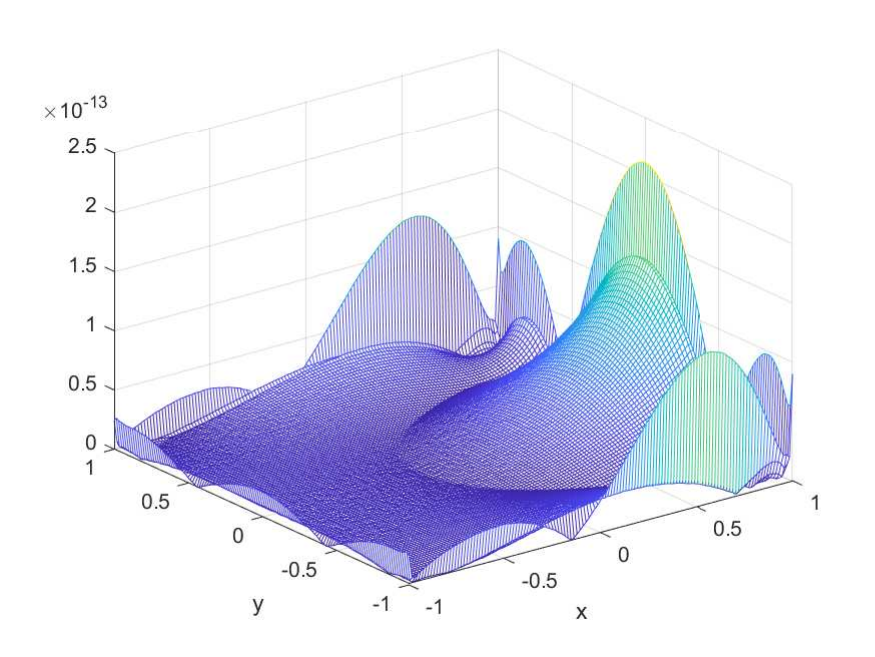}
	\end{subfigure}
	\caption{\cref{example1}:  $p_x$ (first panel), $|(p_x)_h-p_x|$ (second panel), $p_y$ (third panel), and $|(p_y)_h-p_y|$ (fourth panel) at all grid points in $\overline{\Omega}$ with $h=\tfrac{1}{2^6}$ and $\nu=10^{-6}$.}
	\label{fig:p:example:1}
\end{figure}	
\begin{example}\label{example2}
	\normalfont
To demonstrate that the derivations of our sixth-order FDM in \cref{FDMs:all:points} and reformulated PDEs in \cref{thm:all:pde:2D} (the biharmonic equation and third-order PDEs)  do not require  
$\Omega$ to be simply connected,  we consider a triply connected domain.
	Let  $\Omega=(-1,1)^2 \setminus (\cup_{i=1}^3\Omega_i )$ with $\Omega_1=[-1/2,1/4]\times [-1/2,-1/4]$, $\Omega_2=[1/2,3/4]\times [-3/4,1/2]$,  $\Omega_3=[-3/4,0]\times [0,3/4]$,  (see  Figures \ref{fig:u:1:example:2}--\ref{fig:u:2:example:2}) and
	the functions in \eqref{Model:Original} are given by
	\begin{align*}
&	\bu=\left( -\cos(4x) \sin(6y), \sin(4x)\cos(6y) \right ), \qquad \nu>0,\\
& p=\lambda \exp(\tfrac{x}{2}+3y)\text{ with } \lambda=1,10^{10},\quad \text{and} \quad p=\tfrac{\ln(x)}{x^2-1}\tfrac{\ln(y)}{y^2-1},
\end{align*}
and ${\bm f}, \phi, {\bm g}$ can be obtained by plugging the above functions into \eqref{Model:Original}.
	The numerical results are presented in Table \ref{tab:1:example:2} and Figures \ref{fig:u:1:example:2}--\ref{fig:u:2:example:2}. According to the results, our proposed FDM still maintains  sixth-order accuracy on the triply connected domain. 
	%
	As expected, Table \ref{tab:1:example:2} also numerically verifies that the error of $\|\bu_h-\bu\|_{\infty}$ is independent of  $\nu$ and $p$.
\end{example}
\begin{table}[htbp]
	\caption{Performance in \cref{example2}  of the proposed FDM.}
	\centering
	\setlength{\tabcolsep}{1mm}{
		 \begin{tabular}{c|c|c|c|c|c|c|c|c}
			\hline
			\multicolumn{1}{c|}{}  &
			\multicolumn{8}{c}{any $\nu>0$}  \\
			\hline
			\multicolumn{1}{c|}{}  &
			 \multicolumn{2}{c|}{$p=\exp(\tfrac{x}{2}+3y)$}  &
			 \multicolumn{2}{c|}{$p=10^{10}\exp(\tfrac{x}{2}+3y)$}  &
			 \multicolumn{2}{c|}{$p=\tfrac{\ln(x)}{x^2-1}\tfrac{\ln(y)}{y^2-1}$} &
			\multicolumn{2}{c}{any $p$} \\
			\hline
			$h$
			&  $\|\bu_h-\bu\|_{\infty}$
			&order &  $\|\bu_h-\bu\|_{\infty}$ & order &  $\|\bu_h-\bu\|_{\infty}$
			&order  & $\kappa$ & ratio  of $\kappa$ \\
			\hline
$\tfrac{1}{2^3}$ & 1.3318E-01 &  & 1.3318E-01 &  & 1.3318E-01 &  & 1.81E+03 & \\
$\tfrac{1}{2^4}$ & 1.6349E-03 & 6.35 & 1.6349E-03 & 6.35 & 1.6349E-03 & 6.35 & 4.33E+04 & 23.9\\
$\tfrac{1}{2^5}$ & 2.3467E-05 & 6.12 & 2.3467E-05 & 6.12 & 2.3467E-05 & 6.12 & 5.80E+05 & 13.4\\
$\tfrac{1}{2^6}$ & 3.6126E-07 & 6.02 & 3.6126E-07 & 6.02 & 3.6126E-07 & 6.02 & 9.68E+06 & 16.7\\
$\tfrac{1}{2^7}$ & 5.2356E-09 & 6.11 & 5.2356E-09 & 6.11 & 5.2356E-09 & 6.11 & 1.57E+08 & 16.2\\		
			\hline
	\end{tabular}}
	\label{tab:1:example:2}
\end{table}	
\begin{figure}[htbp]
	\centering
	 \begin{subfigure}[b]{0.3\textwidth}
	 \includegraphics[width=5cm,height=5cm]{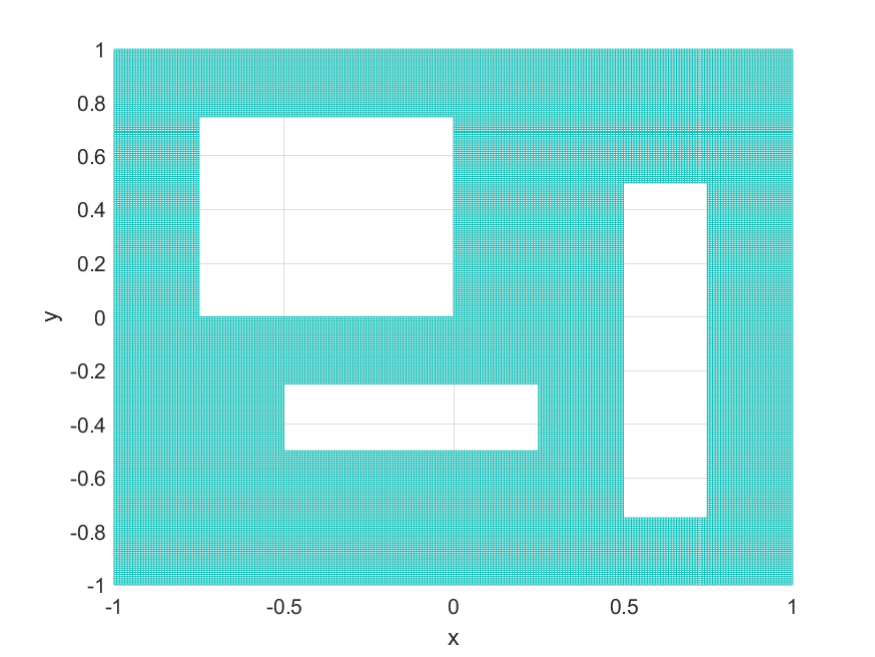}
\end{subfigure}	
	 \begin{subfigure}[b]{0.3\textwidth}
		 \includegraphics[width=5cm,height=5cm]{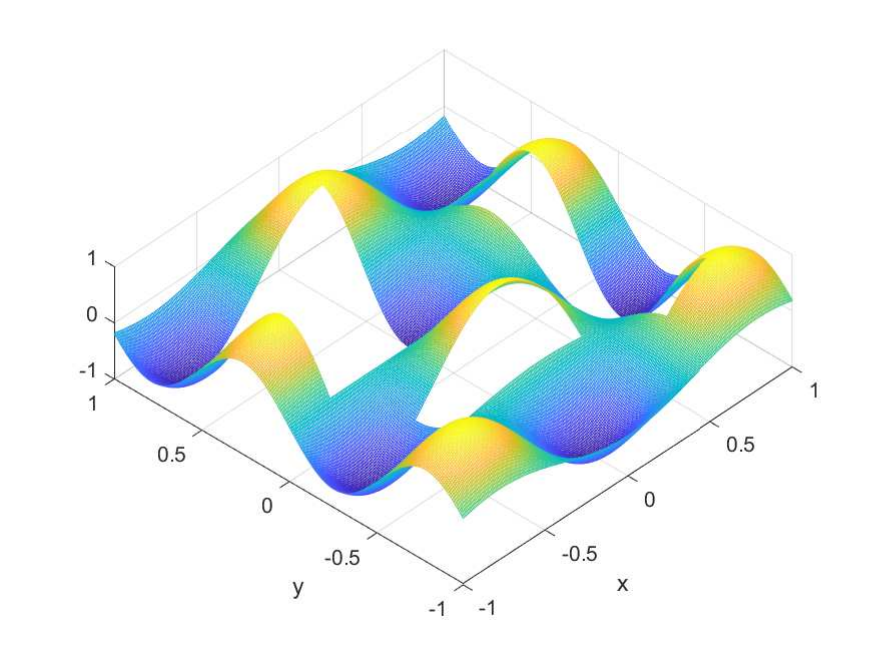}
	\end{subfigure}
	 \begin{subfigure}[b]{0.3\textwidth}
		 \includegraphics[width=5cm,height=5cm]{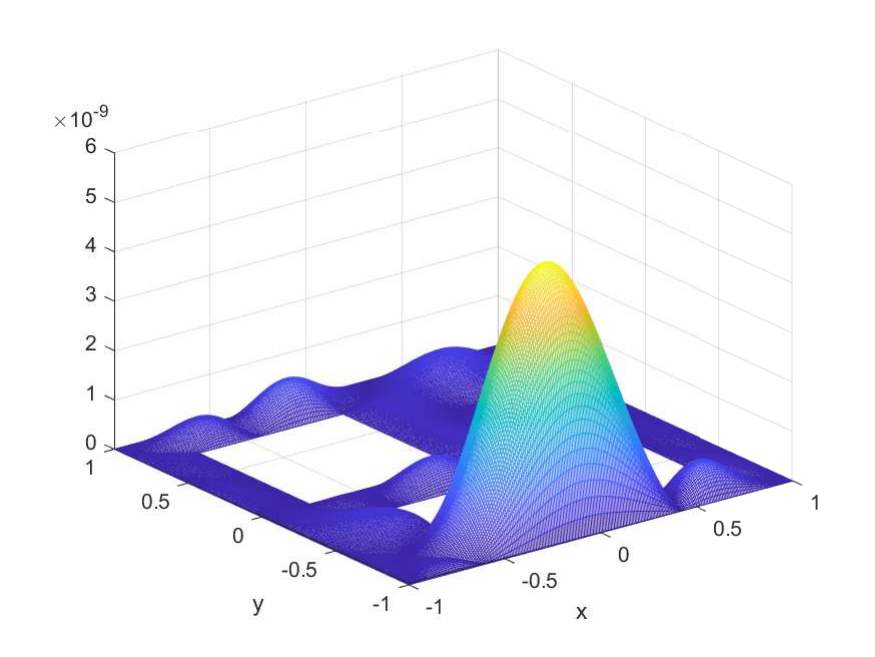}
	\end{subfigure}
	\caption{\cref{example2}: the domain $\Omega$ (left), $u^{(1)}$ (middle),  and $|u_h^{(1)}-u^{(1)}|$ (right) at all grid points in $\overline{\Omega}$ with $h=\tfrac{1}{2^7}$ and any $\nu>0$.}
	\label{fig:u:1:example:2}
\end{figure}	
\begin{figure}[htbp]
	\centering
	 \begin{subfigure}[b]{0.3\textwidth}
		 \includegraphics[width=5cm,height=5cm]{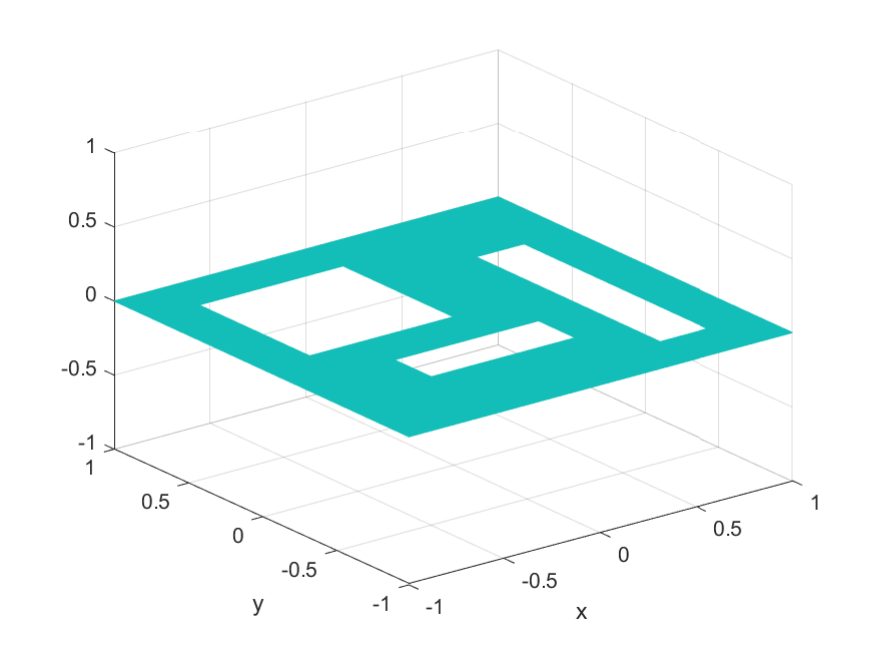}
	\end{subfigure}	
	 \begin{subfigure}[b]{0.3\textwidth}
		 \includegraphics[width=5cm,height=5cm]{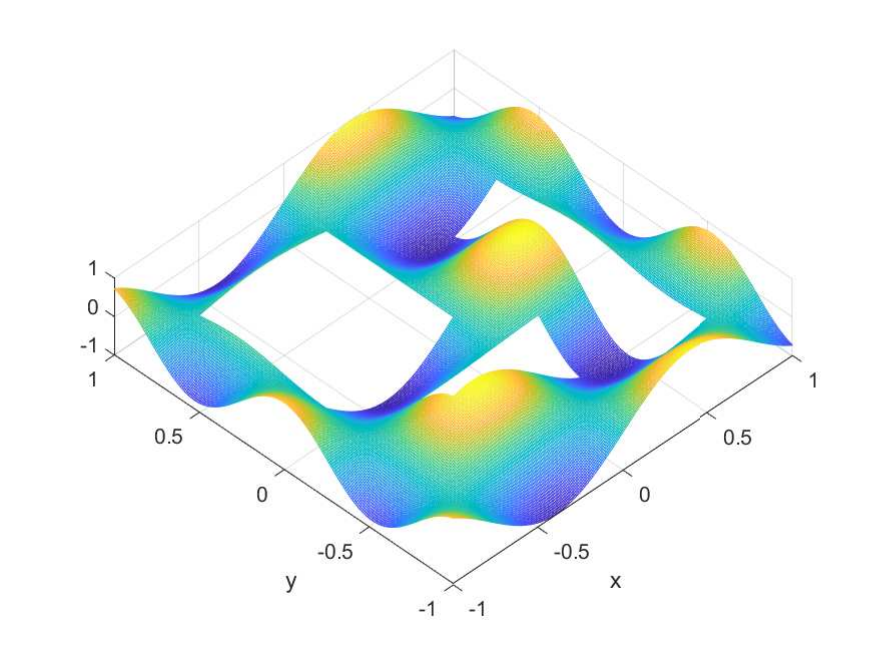}
	\end{subfigure}
	 \begin{subfigure}[b]{0.3\textwidth}
		 \includegraphics[width=5cm,height=5cm]{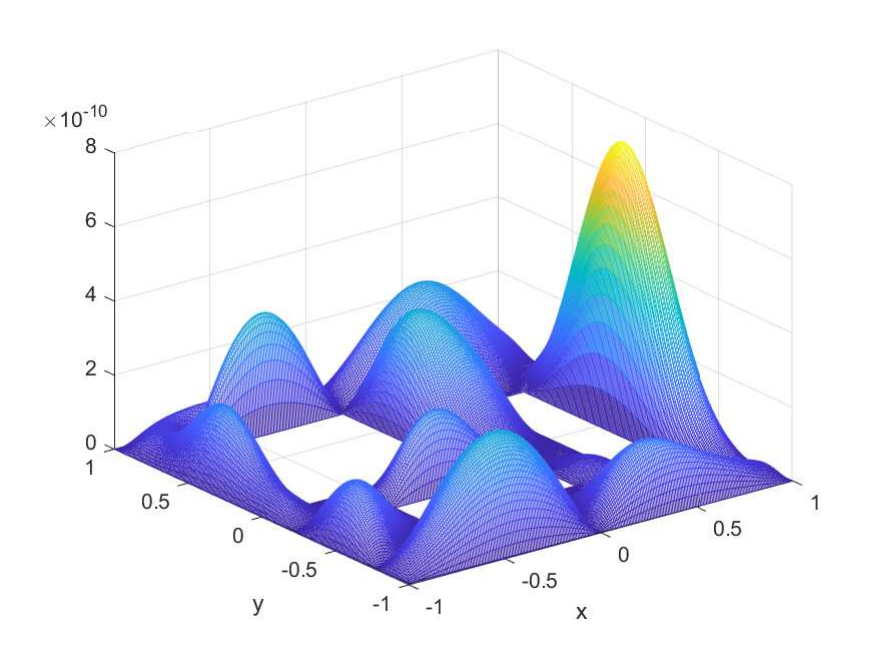}
	\end{subfigure}
	\caption{\cref{example2}: the domain $\Omega$ (left), $u^{(2)}$ (middle),  and $|u_h^{(2)}-u^{(2)}|$ (right) at all grid points in $\overline{\Omega}$ with $h=\tfrac{1}{2^7}$ and any $\nu>0$.}
	\label{fig:u:2:example:2}
\end{figure}	
\begin{example}\label{example3}
	\normalfont
In \cref{example1,example2}, we test smooth velocity $\bu$ and compute  $(u^{(1)}_h,u^{(2)}_h)$ together. In this example, we use \cite[(5.1)--(5.2)]{Brenner2010} and \cite[p.107]{Grisvard1992} to construct a non-smooth 
velocity $\bu$, and solve $(u^{(1)}_h,u^{(2)}_h)$ separately.

	Let  $\Omega=(-1,1)^2\setminus ([0,1) \times (-1,0])$,  (see  Figures \ref{fig:u:1:example:3}--\ref{fig:u:2:example:3}) and
	let the functions in \eqref{Model:Original} be given by
	\begin{align*}
		\bu=( \zeta_y, -\zeta_x ),\qquad p= \exp(x+y),\qquad \phi=0, \qquad \nu>0,
	\end{align*}
	where
	\begin{align*}
\zeta=&(x^2-1)^2(y^2-1)^2\left(\sqrt{x^2+y^2}\right)^{1+z}\eta, \qquad \omega=\tfrac{3\pi}{2},\qquad  z=1.54,\qquad \theta=\arctan(y/x),\\
 \eta=&\left(\tfrac{1}{z-1}\sin((z-1)\omega)-\tfrac{1}{z+1}\sin((z+1)\omega)\right) \left( \cos((z-1)\theta)-\cos((z+1)\theta)\right),\\
 &-\left(\tfrac{1}{z-1}\sin((z-1)\theta)-\tfrac{1}{z+1}\sin((z+1)\theta)\right) \left( \cos((z-1)\omega)-\cos((z+1)\omega)\right),
	\end{align*}
	and ${\bm f},  {\bm g}$ can be obtained by plugging the above functions into \eqref{Model:Original}.
	The numerical results are presented in Table \ref{tab:1:example:3} and Figures \ref{fig:u:1:example:3}--\ref{fig:u:2:example:3}.  The results in Table \ref{tab:1:example:3} confirm the decoupling property of the proposed  reformulated PDEs in \eqref{eqn:FourthOrderPDE}, and the FDM in \eqref{matrix:A1:A2}. By definitions of $\zeta, \bu$ and using \cite[Theorem 1.26 on page 24]{Nicaise1993}, we have $\zeta\in H^{1+z-\epsilon}(\Omega)$ and $\bu\in \bH^{z-\epsilon}(\Omega)$ for any $\epsilon>0$.   The numerical orders in Table \ref{tab:1:example:3} seem to validate the application of 
	the proposed scheme for singular $\bu$. 
	In particular, we see that the rate of convergence for $\|u_h^{(1)}-u^{(1)}\|_\infty$ is approximately $z-1  = 0.54$.

\end{example}
\begin{table}[htbp]
	\caption{Performance in \cref{example3}  of the proposed FDM.}
	\centering
	\setlength{\tabcolsep}{2mm}{
		 \begin{tabular}{c|c|c|c|c|c|c}
			\hline
			\multicolumn{1}{c|}{}  &
			 \multicolumn{6}{c}{any $\nu>0$}  \\
			\hline
			$h$
			&  $\|u_h^{(1)}-u^{(1)}\|_{\infty}$
			&order &  $\|u_h^{(2)}-u^{(2)}\|_{\infty}$  
			&order  & $\kappa$ & ratio  of $\kappa$ \\
			\hline
$\tfrac{1}{2^3}$ & 9.8448E-01 &  & 1.0550E+0 &     &        9.02E+03 & \\
$\tfrac{1}{2^4}$ & 3.7817E-01 & 1.38  & 1.3336E-01 & 2.98        & 1.86E+05 & 20.6\\
$\tfrac{1}{2^5}$ & 2.4121E-01 & 0.65   & 9.6291E-02 & 0.47       & 3.07E+06 & 16.5\\
$\tfrac{1}{2^6}$ & 1.6540E-01 & 0.54  & 6.5650E-02 & 0.55       & 4.90E+07 & 16.0\\
$\tfrac{1}{2^7}$ & 1.1652E-01 & 0.51   & 4.3079E-02 & 0.61      & 7.87E+08 & 16.0\\
$\tfrac{1}{2^8}$ & 8.3405E-02 & 0.48   & 2.7175E-02 & 0.66      & 1.27E+10 & 16.1\\
$\tfrac{1}{2^9}$ & 6.0298E-02 & 0.47   & 1.6400E-02 & 0.73      & 2.03E+11 & 16.1\\			 
			\hline
	\end{tabular}}
\label{tab:1:example:3}
\end{table}	
\begin{figure}[htbp]
	\centering
	 \begin{subfigure}[b]{0.3\textwidth}
		 \includegraphics[width=5cm,height=5cm]{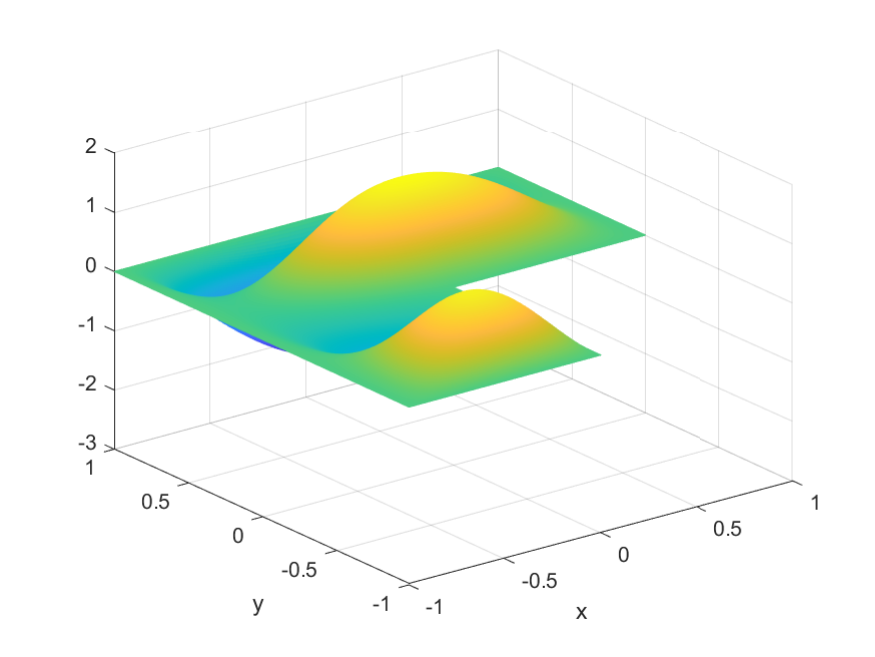}
	\end{subfigure}	
	 \begin{subfigure}[b]{0.3\textwidth}
		 \includegraphics[width=5cm,height=5cm]{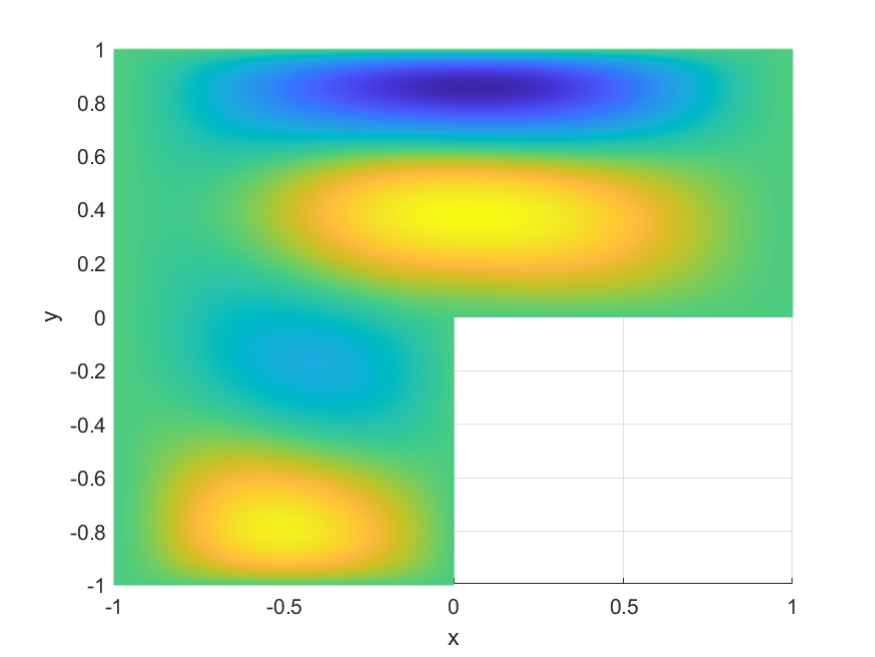}
	\end{subfigure}
	 \begin{subfigure}[b]{0.3\textwidth}
		 \includegraphics[width=5cm,height=5cm]{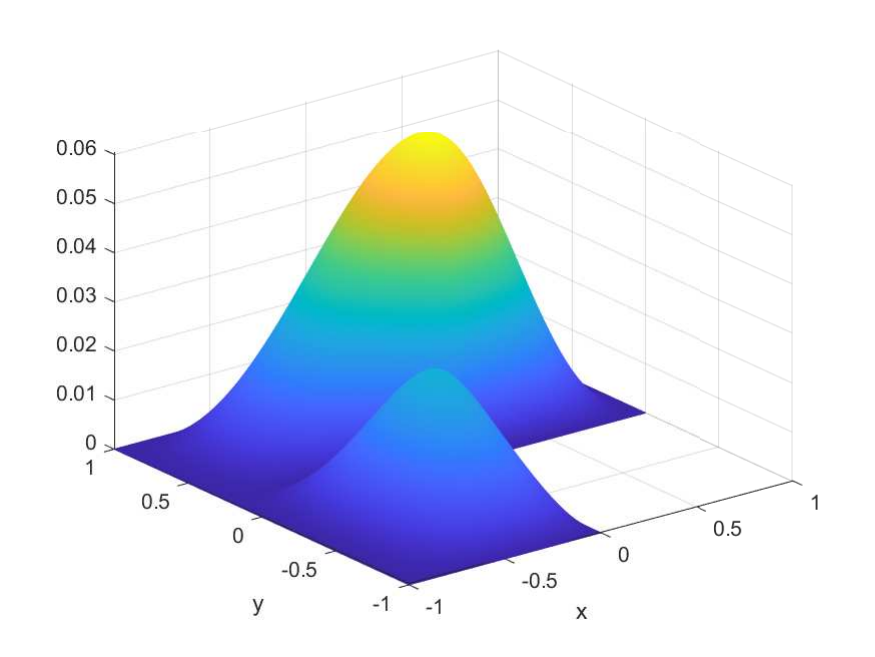}
	\end{subfigure}
	\caption{\cref{example3}:  $u^{(1)}$ (left and middle),  and $|u_h^{(1)}-u^{(1)}|$ (right) at all grid points in $\overline{\Omega}$ with $h=\tfrac{1}{2^9}$ and any $\nu>0$.}
	\label{fig:u:1:example:3}
\end{figure}
\begin{figure}[htbp]
	\centering
	 \begin{subfigure}[b]{0.3\textwidth}
		 \includegraphics[width=5cm,height=5cm]{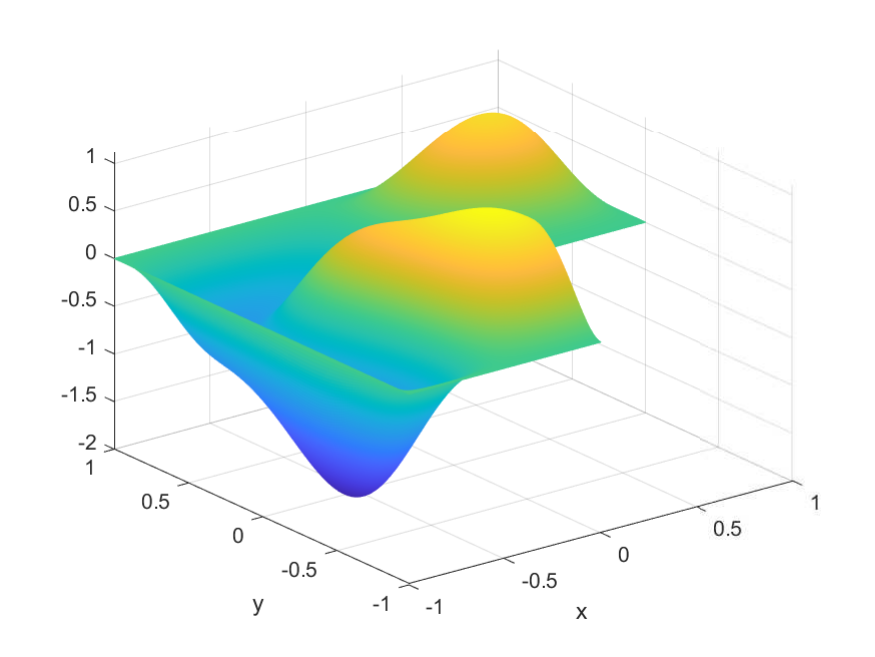}
	\end{subfigure}	
	 \begin{subfigure}[b]{0.3\textwidth}
		 \includegraphics[width=5cm,height=5cm]{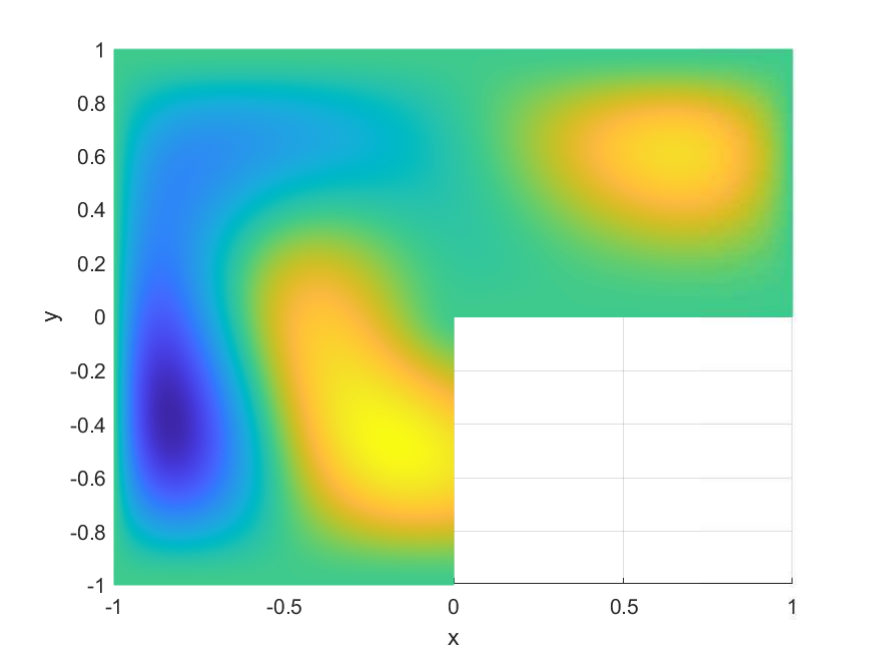}
	\end{subfigure}
	 \begin{subfigure}[b]{0.3\textwidth}
		 \includegraphics[width=5cm,height=5cm]{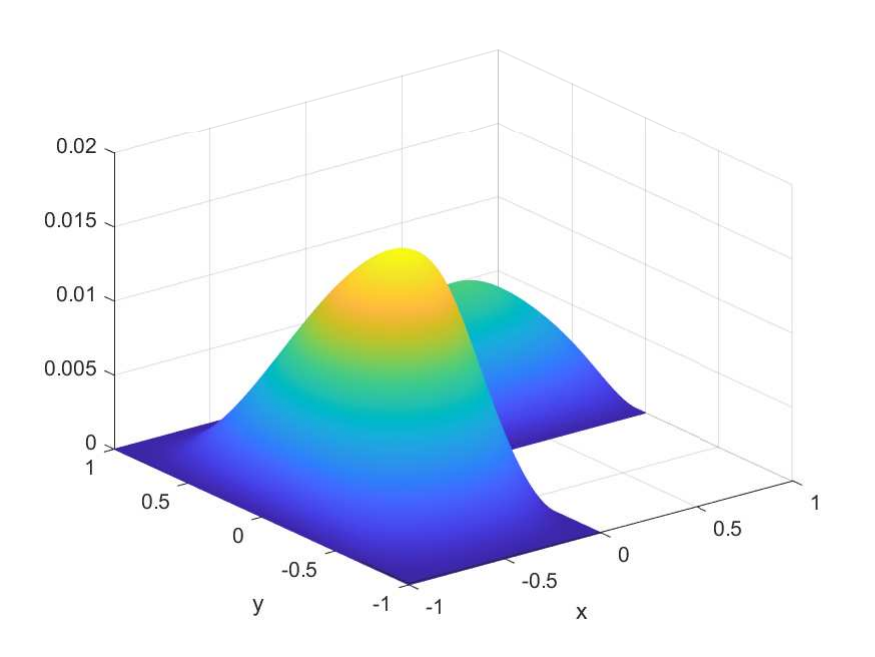}
	\end{subfigure}
	\caption{\cref{example3}:  $u^{(2)}$ (left and middle),  and $|u_h^{(2)}-u^{(2)}|$ (right) at all grid points in $\overline{\Omega}$ with $h=\tfrac{1}{2^9}$ and any $\nu>0$.}
	\label{fig:u:2:example:3}
\end{figure}
\section{Conclusion}\label{sec:Conclu}
In this paper, we developed a finite difference method for the Stokes problem with Dirichlet boundary conditions on an axis-aligned domain. The approach is based on a novel reformulation that leads to a system of fourth-order partial differential equations for the velocity, accompanied by third-order boundary conditions. A key feature of this formulation is that it is independent of both the pressure $p$ and the kinematic viscosity $\nu$, and it does not require the domain $\Omega$ to be simply connected.

For smooth velocity fields $\bu$, we constructed the sixth-order finite difference operator that preserves the decoupling structure at all grid points. Numerical experiments indicate that the scheme is stable
and converges with sixth-order accuracy for smooth solutions. In cases involving singular velocity fields, the method also converges, but at a reduced rate.

Future work will focus on extending the proposed method to more general geometries, including curvilinear and unstructured domains. Additional directions include the adaptation of the method to time-dependent Stokes and Navier-Stokes equations, as well as exploring its extension to three-dimensional problems.

\section{Declarations}
\noindent \textbf{Conflict of interest:} The authors declare that they have no conflict of interest.\\
\noindent \textbf{Data availability:} Data will be made available on reasonable request.

\end{document}